\documentclass[11pt,reqno]{amsart}
\usepackage{amsmath,amsthm,amsfonts,amssymb,mathrsfs}
\date{\today}
\usepackage{color}

\newtheorem{theorem}{Theorem}[section]
\newtheorem{proposition}[theorem]{Proposition}%[section]
\newtheorem{corollary}[theorem]{Corollary}%[section]

\theoremstyle{definition}
\newtheorem{example}[theorem]{Example}%[section]
\newtheorem{remark}[theorem]{Remark}%[section]
\newtheorem{definition}[theorem]{Definition}%[section]

\usepackage{hyperref}

\begin{document}

\title[On a topological simple Warne extension of a semigroup]{On
a topological simple Warne extension of a semigroup}

\author[Iryna~Fihel]{Iryna~Fihel}
\address{Department of Mechanics and Mathematics, Ivan Franko
National University of Lviv, Universytetska 1, Lviv, 79000, Ukraine}
\email{figel.iryna@gmail.com}

\author[Oleg~Gutik]{Oleg~Gutik}
\address{Department of Mechanics and Mathematics, Ivan Franko
National University of Lviv, Universytetska 1, Lviv, 79000, Ukraine}
\email{{o\_gutik}@franko.lviv.ua, {ovgutik}@yahoo.com}

\author[K. Pavlyk]{Kateryna~Pavlyk}
\address{Institute of Mathematics,
University of Tartu, J. Liivi 2, 50409, Tartu, Estonia}
\email{kateryna.pavlyk@ut.ee}


\begin{abstract}
In the paper we introduce topological $\mathbb{Z}$-Bruck-Reilly and
topological $\mathbb{Z}$-Bruck extensions of (semi)topological
monoids which are generalizations of topological Bruck-Reilly and
topological Bruck extensions of (semi)topological monoids and study
their topologizations. The sufficient conditions under which the
topological $\mathbb{Z}$-Bruck-Reilly ($\mathbb{Z}$-Bruck) extension
admits only the direct sum topology and conditions under which the
direct sum topology can be coarsened are given. Also, topological
characterizations of some classes of $I$-bisimple (semi)topo\-lo\-gi\-cal
semigroups are given.
\end{abstract}

\maketitle
%\tableofcontents

\section{Introduction and preliminaries}

In this paper all topological spaces are assumed to be Hausdorff. We
shall follow the terminology of~\cite{CHK, CP, Engelking1989,
Ruppert1984}. If $Y$ is a subspace of a topological space $X$ and
$A\subseteq Y$, then by $\operatorname{cl}_Y(A)$ we shall denote the
topological closure of $A$ in $Y$. Later by $\mathbb{N}$ we denote
the set of positive integers. Also for a map $\theta\colon
X\rightarrow Y$ and positive integer $n$ we denote by
$\theta^{-1}(A)$ and $\theta^{n}(B)$ the full preimage of the set
$A\subseteq Y$ and the $n$-power image of the set $B\subseteq X$,
respectively, i.e., $\theta^{-1}(A)=\{x\in X\colon \theta(x)\in A\}$
and $\theta^{n}(B)=
\{(\underbrace{\theta\circ\ldots\circ\theta}_{n-\hbox{time}})(x)\colon
x\in B\}$.

A semigroup $S$ is \emph{regular} if $x\in xSx$ for every $x\in S$.
A semigroup $S$ is called {\it inverse} if for any element $x\in S$
there exists the unique $x^{-1}\in S$ such that $xx^{-1}x=x$ and
$x^{-1}xx^{-1}=x^{-1}$. The element $x^{-1}$ is called the {\it
inverse of} $x\in S$. If $S$ is an inverse semigroup, then the
function $\operatorname{inv}\colon S\to S$ which assigns to every
element $x$ of $S$ its inverse element $x^{-1}$ is called an {\it
inversion}. An inverse semigroup $S$ is said to be \emph{Clifford}
if $x\cdot x^{-1}=x^{-1}\cdot x$ for all $x\in S$.

If $S$ is a semigroup, then we shall denote the subset of
idempotents in $S$ by $E(S)$. If $S$ is an inverse semigroup, then
$E(S)$ is closed under multiplication and we shall refer to $E(S)$ as a
\emph{band} (or the \emph{band of} $S$). If the band $E(S)$ is a
non-empty subset of $S$, then the semigroup operation on $S$
determines the following partial order $\leqslant$ on $E(S)$:
$e\leqslant f$ if and only if $ef=fe=e$. This order is called the
{\em natural partial order} on $E(S)$. A \emph{semilattice} is a
commutative semigroup of idempotents. A semilattice $E$ is called
{\em linearly ordered} or a \emph{chain} if its natural order is a
linear order. If $E$ is a semilattice and $e\in E$ then we denote
${\downarrow} e=\{ f\in E\mid f\leqslant e\}$ and ${\uparrow} e=\{
f\in E\mid e\leqslant f\}$.

If $S$ is a semigroup, then by $\mathscr{R}$, $\mathscr{L}$, $\mathscr{J}$, $\mathscr{D}$ and $\mathscr{H}$ we shall denote the Green relations on
$S$ (see \cite[Section~2.1]{CP}). A semigroup $S$ is called \emph{simple} if $S$ does not contain any proper two-sided ideals and \emph{bisimple} if $S$ has only one
$\mathscr{D}$-class.

A {\it semitopological} (resp., \emph{topological}) {\it semigroup}
is a Hausdorff topological space together with a separately (resp.,
jointly) continuous semigroup operation \cite{CHK, Ruppert1984}. An
inverse topological semigroup with continuous inversion is
called a \emph{topological inverse semigroup}. A topology $\tau$ on
a (inverse) semigroup $S$ which turns $S$ into a topological
(inverse) semigroup is called a  \emph{semigroup} (\emph{inverse})
\emph{topology} on $S$. A {\it semitopological group} is a Hausdorff
topological space together with a separately continuous group
operation \cite{Ruppert1984} and a \emph{topological group} is a Hausdorff
topological space together with a jointly continuous group operation
and inversion \cite{CHK}.

The bicyclic semigroup ${\mathscr{C}}(p,q)$ is the semigroup with
the identity $1$ generated by elements $p$ and $q$ subjected only to
the condition $pq=1$. The bicyclic semigroup is bisimple and every
one of its congruences is either trivial or a group congruence.
Moreover, every non-annihilating homomorphism $h$ of the bicyclic
semigroup is either an isomorphism or the image of
${\mathscr{C}}(p,q)$ under $h$ is a cyclic group~(see
\cite[Corollary~1.32]{CP}). The bicyclic semigroup plays an
important role in algebraic theory of semigroups and in the theory
of topological semigroups. For example the well-known Andersen's
result~\cite{Andersen} states that a ($0$--)simple semigroup is
completely ($0$--)simple if and only if it does not contain the
bicyclic semigroup. The bicyclic semigroup admits only the discrete
semigroup topology and a topological semigroup $S$ can contain the
bicyclic semigroup ${\mathscr C}(p,q)$ as a dense subsemigroup only
as an open subset~\cite{EberhartSelden1969}. Also Bertman and West
in \cite{BertmanWest1976} proved that the bicyclic semigroup as a
Hausdorff semitopological semigroup admits only the discrete
topology. The problem of an embedding of the bicycle semigroup into
compact-like topological semigroups was solved in the papers \cite{AHK,
BanakhDimitrovaGutik2009, BanakhDimitrovaGutik2010, GutikRepovs2007,
HildebrantKoch1988} and the closure of the bicyclic semigroup in
topological semigroups studied in~\cite{EberhartSelden1969}.

The properties of the bicyclic semigroup were extended to the following two directions: bicyclic-like semigroups which are bisimple and bicyclic-like extensions of semigroups. In the first case such are inverse bisimple semigroups with well-ordered subset of idempotents: $\omega^n$-bisimple semigroups \cite{Hogan1972}, $\omega^\alpha$-bisimple semigroups \cite{Hogan1973} and an $\alpha$-bicyclic semigroup, and bisimple inverse semigroups with linearly ordered subsets of idempotents which are isomorphic to either $[0,\infty)$ or $(-\infty,\infty)$ as  subsets of the real line: $B_{[0,\infty)}^1$, $B_{[0,\infty)}^2$, $B_{(-\infty,\infty)}^1$ and $B_{(-\infty,\infty)}^2$. Ahre \cite{Ahre1981, Ahre1983, Ahre1984, Ahre1986, Ahre1989} and Korkmaz \cite{Korkmaz1997,Korkmaz2009} studied Hausdorff semigroup topologizations of the semigroups $B_{[0,\infty)}^1$, $B_{[0,\infty)}^2$, $B_{(-\infty,\infty)}^1$, and $B_{(-\infty,\infty)}^2$ and their closures in topological semigroups. Annie Selden \cite{Selden1985} and Hogan \cite{Hogan1984} proved that the only locally compact Hausdorff topology making an $\alpha$-bicyclic semigroup into a topological semigroup is the discrete topology. In \cite{Hogan1987} Hogan studied Hausdorff inverse semigroup topologies on an $\alpha$-bicyclic semigroup and there he constructed non-discrete Hausdorff inverse semigroup topology on an $\alpha$-bicyclic semigroup.

Let $\mathbb{Z}$ be the additive group of integers. On the Cartesian
product $\mathscr{C}_{\mathbb{Z}}=\mathbb{Z}\times\mathbb{Z}$ we
define the semigroup operation as follows:
\begin{equation}\label{f1}
    (a,b)\cdot(c,d)=
\left\{
  \begin{array}{ll}
    (a-b+c,d), & \hbox{if }~b<c; \\
    (a,d),     & \hbox{if }~b=c; \\
    (a,d-c+b), & \hbox{if }~b>c,\\
  \end{array}
\right.
\end{equation}
for $a,b,c,d\in\mathbb{Z}$. The set $\mathscr{C}_{\mathbb{Z}}$ with
such defined operation is called the \emph{extended bicyclic
semigroup} \cite{Warne1968}. It is obvious that the extended
bicyclic semigroup is an extension of the bicyclic semigroup. The
extended bicyclic semigroup admits only the discrete topology as a
semitopological semigroup \cite{FihelGutik2011}. Also, the problem
of a closure of $\mathscr{C}_{\mathbb{Z}}$ in topological semigroup was
studied in \cite{FihelGutik2011}.

The conception of Bruck-Reilly extensions started from the Bruck paper \cite{Bruck1958} where he proposed the construction of an embedding of semigroups into simple monoids. Reilly in \cite{Reilly1966} generalized the Bruck construction up to so called in our time Bruck-Reilly construction and using it described the structure of $\omega$-bisimple semigroups. Annie Selden in \cite{Selden1975, Selden1976, Selden1977} described the structure of locally compact topological inverse $\omega$-bisimple semigroups and their closure in topological semigroups.

The disquisition of topological Bruck-Reilly extensions of topological and semitopological semigroups was started in the papers \cite{Gutik1994, Gutik1997} and continued in \cite{Pavlyk2008, GutikPavlyk2009}. Using the ideas of the paper \cite{Gutik1994} Gutik in \cite{Gutik1996} proposed the construction of embedding of an arbitrary topological (inverse) semigroup into a simple path-connected  topological (inverse) monoid.

Let $G$ be a linearly ordered group and $S$ be any semigroup. Let
$\alpha\colon G^+\rightarrow \operatorname{End}(S^1)$ be a
homomorphism from the positive cone $G^+$ into the semigroup of all
endomorphisms of $S^1$. By $\mathscr{B}(S,G,\alpha)$ we denote the
set $G\times S^1\times G$ with the following binary operation
\begin{equation}\label{f2}
	\begin{array}{ll}
    (g_1,s_1,h_1)\cdot(g_2,s_2,h_2)=\\
    =(g_1(h_1{\wedge} g_2)^{-1}g_2,\,
    \alpha[e{\vee} h_1^{-1}g_2](s_1){\cdot}
    \alpha[e{\vee} g_2^{-1}h_1](s_2),\,
    h_2(h_1{\wedge} g_2)^{-1}h_1).
    \end{array}
\end{equation}
A binary operation so defined on the set $\mathscr{B}(S,G^+,\alpha)=G^+\times
S^1\times G^+$ with the semigroup operation induced from $\mathscr{B}(S,G,\alpha)$
 is a subsemigroup of $\mathscr{B}(S,G,\alpha)$ \cite{Fortunatov1977}.

Now we let $G=\mathbb{Z}$ be the additive group of integers with
usual order $\leq$ and $S$ be any semigroup. Let $\alpha\colon
\mathbb{Z}^+\rightarrow \operatorname{End}(S^1)$ be a homomorphisms
from the positive cone $\mathbb{Z}^+$ into the semigroup of all
endomorphisms of $S^1$. Then formula (\ref{f2}) determines the
following semigroup operation on $\mathscr{B}(S,\mathbb{Z},\alpha)$:
\begin{equation*}
   	\begin{array}{ll}
   (i,s,j)\cdot(m,t,n)=\\
   \!\!\!(i{+}m{-}{\min}\{j,m\},
   \alpha[m{-}\min\{j,m\}](s){\cdot}\alpha[j{-}{\min}\{j,m\}](t),
   j{+}n{-}{\min}\{j,m\}),
   	\end{array}
\end{equation*}
where $s,t\in S^1$ and $i,j,m,n\in\mathbb{Z}$.

Let $\theta\colon S^1\rightarrow H(1_S)$ be a homomorphism from the
monoid $S^1$ into the group of units $H(1_S)$ of $S^1$. Then we put
$\alpha[n](s)=\theta^n(s)$, for a positive integer $n$ and
$\theta^0\colon S^1\rightarrow S^1$ be an identity map of $S^1$.
Later the semigroup $\mathscr{B}(S,\mathbb{Z},\alpha)$ with such a
defined homomorphism $\alpha$ we shall denote by
$\mathscr{B}(S,\mathbb{Z},\theta)$, and in the case when the
homomorphism $\theta\colon S^1\rightarrow H(1_S)$ is defined by the
formula
\begin{equation*}
\theta^n(s)= \left\{
  \begin{array}{ll}
    1_S, & \hbox{if~} n>0;\\
    s, & \hbox{if~} n=0,
  \end{array}
\right.
\end{equation*}
we shall denote it by $\mathscr{B}(S,\mathbb{Z})$. We observe that
the semigroup operation on $\mathscr{B}(S,\mathbb{Z},\theta)$ is
defined by the formula
\begin{equation}\label{f3}
    (i,s,j)\cdot(m,t,n)=
\left\{
  \begin{array}{ll}
    (i-j+m,\theta^{m-j}(s)\cdot t,n), & \hbox{if }~j<m; \\
    (i,s\cdot t,n),     & \hbox{if }~j=m; \\
    (i,s\cdot\theta^{j-m}(t),n-m+j), & \hbox{if }~j>m,\\
  \end{array}
\right.
\end{equation}
for $i,j,m,n\in\mathbb{Z}$ and $s,t\in S^1$. Later we shall call the
semigroup $\mathscr{B}(S,\mathbb{Z},\theta)$ the
\emph{$\mathbb{Z}$-Bruck-Reilly extension of the semigroup} $S$ and
$\mathscr{B}(S,\mathbb{Z})$ the \emph{$\mathbb{Z}$-Bruck
extension of the semigroup} $S$, respectively. Also we observe that
if $S$ is a trivial semigroup then the semigroups
$\mathscr{B}(S,\mathbb{Z},\theta)$ and $\mathscr{B}(S,\mathbb{Z})$
are isomorphic to the extended bicyclic semigroup (see
\cite{Warne1968}).

\begin{proposition}\label{proposition-1.1}
Let $S^1$ be a monoid and $\theta\colon S^1\rightarrow H(1_S)$ be a
homomorphism from $S^1$ into the group of units $H(1_S)$ of $S^1$.
The the following statements holds:
\begin{itemize}
  \item[$(i)$] $\mathscr{B}(S,\mathbb{Z},\theta)$ and
   $\mathscr{B}(S,\mathbb{Z})$ are simple semigroups;

  \item[$(ii)$] $\mathscr{B}(S,\mathbb{Z},\theta)$
  ($\mathscr{B}(S,\mathbb{Z})$) is an inverse semigroup if and
   only if $S^1$ is an inverse semigroup;

  \item[$(iii)$] $\mathscr{B}(S,\mathbb{Z},\theta)$
  ($\mathscr{B}(S,\mathbb{Z})$) is a regular semigroup if and
   only if $S^1$ is a regular semigroup.
\end{itemize}
\end{proposition}

The proofs of the statements of Proposition~\ref{proposition-1.1}
are similar to corresponding theorems of Section 8.5 of \cite{CP}
and Theorem~5.6.6 of \cite{Howie1995}.

Also, we remark that the semigroups
$\mathscr{B}(S,\mathbb{Z},\theta)$ and $\mathscr{B}(S,\mathbb{Z})$
have similar descriptions of Green's relations to Bruck-Reilly and
Bruck extensions of $S^1$ (see Lemma~8.46 of \cite{CP} and
Theorem~5.6.6$(2)$ of \cite{Howie1995}), and hence the semigroup
$\mathscr{B}(S,\mathbb{Z},\theta)$ (resp., $\mathscr{B}(S,\mathbb{Z})$) is
bisimple if and only if $S^1$ is bisimple.

\begin{remark}\label{remark-1.2}
Formula (\ref{f3}) implies that if $(i,s,j)\cdot(m,t,n)=(k,d,l)$ in
the semigroup  $\mathscr{B}(S,\mathbb{Z},\theta)$ then
$k-l=i-j+m-n$.
\end{remark}

For every $m ,n\in\mathbb{Z}$ and $A\subseteq S$ we define
$S_{m,n}=\{(m,s,n)\colon s\in S\}$ and $A_{m,n}=\{(m,s,n)\colon s\in
A\}$.

In this paper we introduce topological the $\mathbb{Z}$-Bruck-Reilly and
topological $\mathbb{Z}$-Bruck extensions of (semi)to\-po\-lo\-gi\-cal
monoids which are generalizations of topological Bruck-Reilly and
topological Bruck extensions of (semi)to\-po\-lo\-gi\-cal monoids and study
their topologizations. The sufficient conditions under which the
topological $\mathbb{Z}$-Bruck-Reilly ($\mathbb{Z}$-Bruck) extension
admits only the direct sum topology and conditions under which the
direct sum topology can be coarsened are given. Also, topological
characterizations of some classes of $I$-bisimple (semi)topological
semigroups are given.

\section{On topological $\mathbb{Z}$-Bruck-Reilly extensions}

Let $S$ be a monoid with a group of units $H(1_S)$. Obviously if one of the following conditions holds:
\begin{itemize}
  \item[$1)$] $H(1_S)$ is a trivial group;
  \item[$2)$] $S$ is congruence-free and $S$ is not a group;
  \item[$3)$] $S$ has zero,
\end{itemize}
then every homomorphism $\theta\colon S^1\rightarrow H(1_S)$ is annihilating. Also, many topological properties of a (semi)topological semigroup $S$ guarantee the triviality of $\theta$. For example, such is the following: $H(1_S)$ is a discrete subgroup of $S$ and $S$ has a minimal ideal $K(S)$ which is a connected subgroup of $S$.

On the other side there exist many conditions on a (semitopological, topological) semigroup $S$ which ensure the existence of a non-annihilating (continuous) homomorphism  $\theta\colon S^1\rightarrow H(1_S)$ from $S$ into non-trivial group of units $H(1_S)$. For example, such conditions are the following:
\begin{itemize}
  \item[$1)$] the (semitopological, topological) semigroup $S$ has a minimal ideal $K(S)$ which is a non-trivial group and there exists a non-annihilating (continuous) homomorphism $h\colon K(S)\rightarrow H(1_S)$;
  \item[$2)$] $S$ is an inverse semigroup and there exists a non-annihilating homomorphism $h\colon S/\sigma\rightarrow H(1_S)$, where $\sigma$ is the least group congruence on $S$ (see \cite[Section~III.5]{Petrich1984}).
\end{itemize}

Let $(S,\tau)$ be a semitopological monoid and $1_S$ be a unit of
$S$. If $S$ does not contain a unit then without loss of generality
we can assume that $S$ is a semigroup with an isolated adjoined
unit. Also we shall assume that the homomorphism $\theta\colon
S^1\rightarrow H(1_S)$ is continuous.

Let $\mathcal{B}$ be a base of the topology $\tau$ on $S$. According
to \cite{Gutik1994} the topology $\tau_\textbf{BR}$ on
$\mathscr{B}(S,\mathbb{Z},\theta)$ generated by the base
\begin{equation*}
    \mathcal{B}_{BR}=\left\{(i,U,j)\colon
    U\in\mathcal{B},i,j\in\mathbb{Z}\right\}
\end{equation*}
is called a \emph{direct sum topology} on
$\mathscr{B}(S,\mathbb{Z},\theta)$ and we shall denote it by
$\tau_{\textbf{BR}}^\textbf{ds}$. We observe that the topology
$\tau_{\textbf{BR}}^\textbf{ds}$ is the product topology on
$\mathscr{B}(S,\mathbb{Z},\theta)= \mathbb{Z}\times
S\times\mathbb{Z}$.

\begin{proposition}\label{proposition-2.1}
Let $(S,\tau)$ be a semitopological (resp., topological, topological
inverse) semigroup and $\theta\colon S^1\rightarrow H(1_S)$ be a
continuous homomorphism from $S$ into the group of units $H(1_S)$
of $S$. Then $\left(\mathscr{B}(S,\mathbb{Z},\theta),
\tau_{\textbf{BR}}^\textbf{ds}\right)$ is a semitopological (resp.,
topological, topological inverse) semigroup.
\end{proposition}

The proof of Proposition~\ref{proposition-2.1} is similar to
Theorem~1 from \cite{Gutik1994}.

\begin{definition}\label{definition-2.2}
Let $\mathfrak{S}$ be some class of semitopological semigroups and
$(S,\tau)\in\mathfrak{S}$. If $\tau_\textbf{BR}$ is a topology on
$\mathscr{B}(S,\mathbb{Z},\theta)$ such that the homomorphism
$\theta\colon S^1\rightarrow H(1_S)$ is a continuous map,
$\left(\mathscr{B}(S,\mathbb{Z},\theta),\tau_\textbf{BR}\right)
\in\mathfrak{S}$ and $\tau_\textbf{BR}|_{S_{m,m}}=\tau$ for some
$m\in\mathbb{Z}$, then the semigroup
$\left(\mathscr{B}(S,\mathbb{Z},\theta),\tau_\textbf{BR}\right)$ is
called a \emph{topological $\mathbb{Z}$-Bruck-Reilly extension} of
the semitopological semigroup $(S,\tau)$ in the class
$\mathfrak{S}$. In the case when $\theta(s)=1_S$ for all $s\in S^1$,
the semigroup $\left(\mathscr{B}(S,\mathbb{Z}),
\tau_\textbf{BR}\right)$ is called a \emph{topological
$\mathbb{Z}$-Bruck extension} of the semitopological semigroup
$(S,\tau)$ in the class $\mathfrak{S}$.
\end{definition}

Proposition~\ref{proposition-2.1} implies that for every
semitopological (resp., topological, topological inverse) semigroup
$(S,\tau)$ there exists a \emph{topological
$\mathbb{Z}$-Bruck-Reilly extension}
$\left(\mathscr{B}(S,\mathbb{Z},\theta),
\tau_{\textbf{BR}}^\textbf{ds}\right)$ of the semitopological
(resp., topological, topological inverse) semigroup $(S,\tau)$ in
the class of semitopological (resp., topological, topological
inverse) semigroups. It is natural to ask: \emph{when is
$\left(\mathscr{B}(S,\mathbb{Z},\theta),
\tau_{\textbf{BR}}^\textbf{ds}\right)$ unique for the semigroup
$(S,\tau)$?}

\begin{proposition}\label{proposition-2.3}
Let $\left(\mathscr{B}(S,\mathbb{Z},\theta),
\tau_{\textbf{BR}}\right)$ be a semitopological semigroup. Then the
following conditions hold:
\begin{itemize}
 \item[$(i)$] for every $i,j,k,l\in\mathbb{Z}$ the topological
  subspaces $S_{i,j}$ and $S_{k,l}$ are homeomorphic and moreover
  $S_{i,i}$ and $S_{k,k}$ are topologically isomorphic
  subsemigroups in $\left(\mathscr{B}(S,\mathbb{Z},\theta),
  \tau_{\textbf{BR}}\right)$;

 \item[$(ii)$] for every $(i,s,j) \in
  \mathscr{B}(S,\mathbb{Z},\theta)$ there exists an open neighbourhood
  $U_{(i,s,j)}$ of the point $(i,s,j)$ in
  $\left(\mathscr{B}(S,\mathbb{Z},\theta),
  \tau_{\textbf{BR}}\right)$ such that
  \begin{equation*} U_{(i,s,j)}\subseteq\bigcup\left\{S_{i-k,j-k}\colon
  k=0,1,2,3,\ldots\right\}. \end{equation*}
\end{itemize}
\end{proposition}

\begin{proof}
$(i)$ For every $i,j,k,l\in\mathbb{Z}$ the map
$\phi^{k,l}_{i,j}\colon \mathscr{B}(S,\mathbb{Z},\theta)\rightarrow
\mathscr{B}(S,\mathbb{Z},\theta)$ defined by the formula
$\phi^{k,l}_{i,j}(x)=(k,1_S,i)\cdot x\cdot(j,1_S,l)$ is continuous
as a composition of left and right translations in the
semitopological semigroup $\left(\mathscr{B}(S,\mathbb{Z},\theta),
\tau_{\textbf{BR}}\right)$. Since
$\phi^{i,j}_{k,l}(\phi^{k,l}_{i,j}(s))=s$ and
$\phi^{k,l}_{i,j}(\phi^{i,j}_{k,l}(t))=t$ for all $s\in S_{i,j}$ and
$t\in S_{k,l}$ we conclude that the restriction
$\phi^{k,l}_{i,j}|_{S_{i,j}}$ is the inverse map of the restriction
$\phi^{i,j}_{k,l}|_{S_{k,l}}$. Then the continuity of the map
$\phi^{k,l}_{i,j}$ implies that the restriction
$\phi^{k,l}_{i,j}|_{S_{i,j}}$ is a homeomorphism which maps elements
of the subspace $S_{i,j}$ onto elements of the subspace $S_{k,l}$ in
$\mathscr{B}(S,\mathbb{Z},\theta)$. Now the definition of the map
$\phi^{k,l}_{i,j}$ implies that the restriction
$\phi^{k,k}_{i,i}|_{S_{i,i}}\colon S_{i,i}\rightarrow S_{k,k}$ is a
topological isomorphism of semitopological subsemigroups $S_{i,i}$
and $S_{k,k}$.

$(ii)$ Since the left and right translations in a semitopological
semigroup are continuous maps and left and right translations by an
idempotent are retractions, Exercise~1.5.C from \cite{Engelking1989}
implies that $(i+1,1_S,i+1)\mathscr{B}(S,\mathbb{Z},\theta)$ and
$\mathscr{B}(S,\mathbb{Z},\theta)(j+1,1_S,j+1)$ are closed subsets
in $\left(\mathscr{B}(S,\mathbb{Z},\theta),
\tau_{\textbf{BR}}\right)$. Hence there exists an open neighbourhood
$W_{(i,s,j)}$ of the point $(i,s,j)$ in
$\left(\mathscr{B}(S,\mathbb{Z},\theta), \tau_{\textbf{BR}}\right)$
such that
\begin{equation*}
W_{(i,s,j)}{\subseteq}\mathscr{B}(S,\mathbb{Z},\theta)\setminus
\left((i{+}1,1_S,i{+}1)\mathscr{B}(S,\mathbb{Z},\theta)\cup
\mathscr{B}(S,\mathbb{Z},\theta)(j{+}1,1_S,j{+}1)\right).
\end{equation*}
Since the semigroup operation in
$\left(\mathscr{B}(S,\mathbb{Z},\theta), \tau_{\textbf{BR}}\right)$
is separately continuous we conclude that there exists an open
neighbourhood $U_{(i,s,j)}$ of the point $(i,s,j)$ in
$\left(\mathscr{B}(S,\mathbb{Z},\theta), \tau_{\textbf{BR}}\right)$
such that
\begin{equation*}
    U_{(i,s,j)}\subseteq W_{(i,s,j)}, \;
    (i,1_S,i)\cdot U_{(i,s,j)}\subseteq W_{(i,s,j)} \;
    \hbox{and} \;
    U_{(i,s,j)}\cdot(j,1_S,j)\subseteq W_{(i,s,j)}.
\end{equation*}
Next we shall show that $U_{(i,s,j)}\subseteq
\bigcup\left\{S_{i-k,j-k}\colon k=0,1,2,3,\ldots\right\}$. Suppose
the contrary: there exists $(m,a,n)\in U_{(i,s,j)}$ such that
$(m,a,n)\notin\bigcup\left\{S_{i-k,j-k}\colon
k=0,1,2,3,\ldots\right\}$. Then we have that $m\leqslant i$,
$n\leqslant j$ and $m-n\neq i-j$. If $m-n>i-j$ then we have that
\begin{equation*}
(m{,}a{,}n){\cdot} (j,1_S,j){=}
(m{-}n{+}j,\theta^{j{-}n}(a),j)
{\notin}\mathscr{B}(S,\mathbb{Z},\theta){\setminus}
(i{+}1,1_S,i{+}1)\mathscr{B}(S,\mathbb{Z},\theta),
\end{equation*}
because $m-n+j>i-j+j=i$, and hence $(m,a,n)\cdot (j,1_S,j)\notin
W_{(i,s,j)}$. Similarly if $m-n<i-j$ then we have that
\begin{equation*}
(i{,}{1_S}{,}i){\cdot}(m,a,n){=}
(i,\theta^{i{-}m}(a),n{-}m{+}i)
{\notin}\mathscr{B}(S,\mathbb{Z},\theta)\setminus
\mathscr{B}(S,\mathbb{Z},\theta)(j{+}1,{1_S},j{+}1),
\end{equation*}
because $n-m+i>j-i+i=j$, and hence $(i,1_S,i)\cdot(m,a,n)\notin
W_{(i,s,j)}$. This completes the proof of our statement.
\end{proof}

\begin{theorem}\label{theorem-2.4}
Let $\left(\mathscr{B}(S,\mathbb{Z},\theta),
\tau_{\textbf{BR}}\right)$ be a topological
$\mathbb{Z}$-Bruck-Reilly extension of the semitopological semigroup
$(S,\tau)$. If $S$ contains a left (right or two-sided) compact
ideal, then $\tau_{\textbf{BR}}$ is the direct sum topology on
$\mathscr{B}(S,\mathbb{Z},\theta)$.
\end{theorem}

\begin{proof}
We consider the case when the semitopological semigroup $S$ has a
left compact ideal. In other cases the proof is similar. Let $L$ be
a left compact ideal in $S$. Then by Definition~\ref{definition-2.2}
there exists an integer $n$ such that the subsemigroup $S_{n,n}$ in
$\left(\mathscr{B}(S,\mathbb{Z},\theta), \tau_{\textbf{BR}}\right)$
is topologically isomorphic to the semitopological semigroup
$(S,\tau)$. Hence Proposition~\ref{proposition-2.3} implies that
$L_{i,j}$ is a compact subset of
$\left(\mathscr{B}(S,\mathbb{Z},\theta), \tau_{\textbf{BR}}\right)$
for all $i,j\in\mathbb{Z}$.

We fix an arbitrary element $(i,s,j)$ of the semigroup
$\mathscr{B}(S,\mathbb{Z},\theta)$, $i,j\in\mathbb{Z}$ and $s\in
S^1$. Now we fix an element $(i-1,t,j-1)$ from $L_{i-1,j-1}$ and
define a map $h\colon \mathscr{B}(S,\mathbb{Z},\theta)\rightarrow
\mathscr{B}(S,\mathbb{Z},\theta)$ by the formula
$h(x)=x\cdot(j-1,t,j-1)$. Then by
Proposition~\ref{proposition-2.3}$(ii)$ there exists an open
neighbourhood $U_{(i,s,j)}$ of the point $(i,s,j)$ in
$\left(\mathscr{B}(S,\mathbb{Z},\theta), \tau_{\textbf{BR}}\right)$
such that $U_{(i,s,j)}\subseteq\bigcup\left\{S_{i-k,j-k}\colon
k=0,1,2,3,\ldots\right\}$. Since left translations in
$\left(\mathscr{B}(S,\mathbb{Z},\theta), \tau_{\textbf{BR}}\right)$
are continuous we conclude that the full pre-image
$h^{-1}(L_{i-1,j-1})$ is a closed subset of the topological space
$\left(\mathscr{B}(S,\mathbb{Z},\theta), \tau_{\textbf{BR}}\right)$
and Remark~\ref{remark-1.2} implies that $h^{-1}(L_{i-1,j-1})=
\bigcup\left\{S_{i-k,j-k}\colon k=1,2,3,\ldots\right\}$. This
implies that an arbitrary element $(i,s,j)$ of the semigroup
$\mathscr{B}(S,\mathbb{Z},\theta)$, where $i,j\in\mathbb{Z}$ and
$s\in S^1$, has an open neighbourhood $U_{(i,s,j)}$ such that
$U_{(i,s,j)}\subseteq S_{i,j}$.
\end{proof}

Theorem~\ref{theorem-2.4} implies the following corollary:

\begin{corollary}[{\cite{FihelGutik2011}}]\label{corollary-2.5}
Let $\tau$ be a Hausdorff topology  on the extended bicyclic
semigroup $\mathscr{C}_{\mathbb{Z}}$. If
$(\mathscr{C}_{\mathbb{Z}},\tau)$ is a semitopological semigroup
then $(\mathscr{C}_{\mathbb{Z}},\tau)$ is the discrete  space.
\end{corollary}

\begin{theorem}\label{theorem-2.6}
Let $\left(\mathscr{B}(S,\mathbb{Z},\theta),
\tau_{\textbf{BR}}\right)$ be a topological
$\mathbb{Z}$-Bruck-Reilly extension of the topological inverse
semigroup $(S,\tau)$ in the class of topological inverse semigroups.
If the band $E(S)$ contains a minimal idempotent, then
$\tau_{\textbf{BR}}$ is the direct sum topology on
$\mathscr{B}(S,\mathbb{Z},\theta)$.
\end{theorem}

\begin{proof}
Let $e_0$ be a minimal element of the band $E(S)$. Then $(i,e_0,i)$
is a minimal idempotent in the band of the subsemigroup $S_{i,i}$
for every integer~$i$.

Since the semigroup operation in
$\left(\mathscr{B}(S,\mathbb{Z},\theta), \tau_{\textbf{BR}}\right)$
is continuous we conclude that for every idempotent $\iota$ from the
semigroup $\mathscr{B}(S,\mathbb{Z},\theta)$ the set
${\uparrow}\iota=\{\varepsilon\in
E(\mathscr{B}(S,\mathbb{Z},\theta))\colon
\varepsilon\cdot\iota=\iota\cdot\varepsilon=\iota\}$ is a closed
subset in $E(\mathscr{B}(S,\mathbb{Z},\theta))$ with the topology
induced from $\left(\mathscr{B}(S,\mathbb{Z},\theta),
\tau_{\textbf{BR}}\right)$. We define the maps $\mathfrak{l}\colon
\mathscr{B}(S,\mathbb{Z},\theta)\rightarrow
E(\mathscr{B}(S,\mathbb{Z},\theta))$ and $\mathfrak{r}\colon
\mathscr{B}(S,\mathbb{Z},\theta)\rightarrow
E(\mathscr{B}(S,\mathbb{Z},\theta))$ by the formulae
$\mathfrak{l}(x)=x\cdot x^{-1}$ and $\mathfrak{r}(x)=x^{-1}\cdot x$.
We fix any element $(i,s,j)\in\mathscr{B}(S,\mathbb{Z},\theta)$.
Since the semigroup operation and inversion are continuous in
$\left(\mathscr{B}(S,\mathbb{Z},\theta), \tau_{\textbf{BR}}\right)$
we conclude that the sets
$\mathfrak{l}^{-1}\left({\uparrow}(i-1,e_0,i-1)\right)$ and
$\mathfrak{r}^{-1}\left({\uparrow}(j-1,e_0,j-1)\right)$ are closed
in $\left(\mathscr{B}(S,\mathbb{Z},\theta),\tau_{\textbf{BR}}\right)$. Then by
Proposition~\ref{proposition-2.3}$(ii)$ there exists an open
neighbourhood $U_{(i,s,j)}$ of the point $(i,s,j)$ in
$\left(\mathscr{B}(S,\mathbb{Z},\theta),
\tau_{\textbf{BR}}\right)$ such that
$U_{(i,s,j)}\subseteq\bigcup\{S_{i-k,j-k}\colon k=0,1,2,3,\ldots\}$.
Now elementary calculations  show that
\begin{equation*}
W_{(i,s,j)}=U_{(i,s,j)}\setminus
\left(\mathfrak{l}^{-1}\left({\uparrow}(i-1,e_0,i-1)\right)\cup
\mathfrak{r}^{-1}\left({\uparrow}(j-1,e_0,j-1)\right)\right)
\subseteq S_{i,j}.
\end{equation*}
This completes the proof of our theorem.
\end{proof}

The following examples show that the arguments stated in
Theorems~\ref{theorem-2.4} and ~\ref{theorem-2.6} are important.

\begin{example}\label{example-2.7}
Let $N_+=\{0,1,2,3,\ldots\}$ be the discrete topological space with
the usual operation of addition of integers. We define a topology
$\tau_{\textbf{BR}}$ on $\mathscr{B}(N_+,\mathbb{Z})$ as follows:
\begin{itemize}
  \item[$(i)$] for every point $x\in N_+\setminus\{0\}$ the base
   of the topology $\tau_{\textbf{BR}}$ at $(i,x,j)$ coincides with the base of
   the direct sum topology $\tau_{\textbf{BR}}^{\textbf{ds}}$ at
   $(i,x,j)$ for all $i,j\in\mathbb{Z}$;
  \item[$(ii)$] for any $i,j\in\mathbb{Z}$ the family
  $\mathscr{B}_{(i,0,j)}=\big\{U^k_{i,j}\colon
  k=1,2,3,\ldots\big\}$, where
\begin{equation*}
 U^k_{i,j}=\{(i,0,j)\}\cup\{(i-1,s,j-1)\colon s=k, k+1,k+2,k+3,\ldots\},
\end{equation*}
   is the base of the topology $\tau_{\textbf{BR}}$ at the point
   $(i,0,j)$.
\end{itemize}
Simple verifications show that
$\left(\mathscr{B}(N_+,\mathbb{Z}),\tau_{\textbf{BR}}\right)$ is a
Hausdorff topological semigroup.
\end{example}

\begin{example}\label{example-2.8}
Let $N_\textbf{m}=\{0,1,2,3,\ldots\}$ be the discrete topological
space with the semigroup operation $x\cdot y=\max\{x,y\}$. Now we
identify the set $\mathscr{B}(N_\textbf{m},\mathbb{Z})$ with
$\mathscr{B}(N_+,\mathbb{Z})$. Let $\tau_{\textbf{BR}}$ be the
topology on $\mathscr{B}(N_+,\mathbb{Z})$ defined as in
Example~\ref{example-2.7}. Then simple verifications show that
$\left(\mathscr{B}(N_\textbf{m},\mathbb{Z}),\tau_{\textbf{BR}}\right)$
is a Hausdorff topological inverse semigroup.
\end{example}

\begin{definition}\label{definition-2.9}
We shall say that a semitopological semigroup $S$ has an \emph{open
ideal property} (or shortly, $S$ is an
\emph{\textsf{OIP}-semigroup}) if there exist a family
$\mathscr{I}=\{I_\alpha\}_{\alpha\in\mathcal{A}}$ of open ideals in
$S$ such that for every $x\in S$ there exist an open ideal
$I_\alpha\in\mathscr{I}$ and open neighbourhood $U(x)$ of the point
$x$ in $S$ such that $U(x)\cap I_\alpha=\varnothing$.
\end{definition}

We observe that Definition~\ref{definition-2.9} implies that the family
$\mathscr{I}=\{I_\alpha\}_{\alpha\in\mathcal{A}}$ of open ideals in
$S$ satisfies the finite intersection property and every
semitopological \textsf{OIP}-semigroup does not contain a minimal
ideal.

\begin{theorem}\label{theorem-2.10}
Let $(S,\tau)$ be a Hausdorff semitopological \textsf{OIP}-semigroup
and $\theta\colon S^1\rightarrow H(1_S)$ be a continuous
homomorphism. Then there exists a topological
$\mathbb{Z}$-Bruck-Reilly extension
$\left(\mathscr{B}(S,\mathbb{Z},\theta), \tau_{\textbf{BR}}\right)$
of $(S,\tau)$ in the class of semitopological semigroups such that
the topology $\tau_{\textbf{BR}}$ is strictly coarser than the direct
sum topology $\tau_{\textbf{BR}}^\textbf{ds}$ on
$\mathscr{B}(S,\mathbb{Z},\theta)$.
\end{theorem}

\begin{proof}
Let $\mathscr{I}=\{I_\alpha\}_{\alpha\in\mathcal{A}}$  be a family
of open ideals in $(S,\tau)$ such that for every $x\in S$ there
exists $I_\alpha\in\mathscr{I}$ and open neighbourhood $U(x)$ of the
point $x$ in $(S,\tau)$ such that $U(x)\cap I_\alpha=\varnothing$.

We shall define a base of the topology $\tau_{\textbf{BR}}$ on
$\mathscr{B}(S,\mathbb{Z},\theta)$ in the following way:
\begin{itemize}
 \item[$(1)$] for every $s\in S\setminus H(1_S)$ and
  $i,j\in\mathbb{Z}$ the base of the topology
  $\tau_{\textbf{BR}}$ at the point $(i,s,j)$ coincides with a
  base of the direct sum topology $\tau_{\textbf{BR}}^\textbf{ds}$
  at $(i,s,j)$; \; and

 \item[$(2)$] the family
 \begin{equation*}
 \mathscr{B}_{(i,a,j)}=\big\{
  (U_a)^\alpha_{i,j}=(U_a)_{i,j}\cup
  \left(\theta^{-1}(U_a)\cap I_\alpha\right)_{i-1,j-1}\colon U_a\in
  \mathscr{B}_{a},I_\alpha\in\mathscr{I}\big\},
 \end{equation*}
  where $\mathscr{B}_{a}$ is a base of the topology $\tau$ at the
  point $a$ in $S$, is a base of the topology $\tau_{\textbf{BR}}$
  at the point $(i,a,j)$, for every $a\in H(1_S)$ and all
  $i,j\in\mathbb{Z}$.
\end{itemize}

Since  $(S,\tau)$ is a Hausdorff semitopological
\textsf{OIP}-semigroup we conclude that $\tau_{\textbf{BR}}$ is a
Hausdorff topology on $\mathscr{B}(S,\mathbb{Z},\theta)$ and
moreover $\tau_{\textbf{BR}}$ is a proper subfamily of
$\tau_{\textbf{BR}}^\textbf{ds}$. Hence $\tau_{\textbf{BR}}$ is a
coarser topology on $\mathscr{B}(S,\mathbb{Z},\theta)$ than
$\tau_{\textbf{BR}}^\textbf{ds}$.

Now we shall show that the semigroup operation on
$(\mathscr{B}(S,\mathbb{Z},\theta),\tau_{\textbf{BR}})$ is
separately continuous. Since by Proposition~\ref{proposition-2.1}
the semigroup operation on
$(\mathscr{B}(S,\mathbb{Z},\theta),\tau_{\textbf{BR}}^\textbf{ds})$
is separately continuous we conclude that the  definition of the
topology $\tau_{\textbf{BR}}$ on $\mathscr{B}(S,\mathbb{Z},\theta)$
implies that it is sufficient to show that the semigroup operation
in $(\mathscr{B}(S,\mathbb{Z},\theta),\tau_{\textbf{BR}})$ is
separately continuous in the following three cases:
\begin{equation*}
1)~(i,h,j)\cdot(m,g,n); \quad ~2)~(i,h,j)\cdot(m,s,n); \quad \hbox{and } \quad
  ~3)~(m,s,n)\cdot(i,h,j),
\end{equation*}
where $s\in S\setminus H(1_S)$, $g,h\in H(1_S)$ and
$i,j,m,n\in\mathbb{Z}$.

Consider case $1)$. Then we have that
\begin{equation*}\label{f7}
 (i,h,j)\cdot(m,g,n)=
 \left\{
   \begin{array}{ll}
     (i-j+m,\theta^{m-j}(h)\cdot g,n), & \hbox{if~} j<m; \\
     (i,h\cdot g,n), & \hbox{if~} j=m; \\
     (i,h\cdot\theta^{j-m}(g),n-m+j), & \hbox{if~} j>m.
   \end{array}
 \right.
\end{equation*}

Suppose that $j<m$. Then the separate continuity of the semigroup
operation in $(S,\tau)$ and the continuity of the homomorphism
$\theta\colon S\rightarrow H(1_S)$ imply that for every open
neighbourhood $U_{\theta^{m-j}(h)\cdot g}$ of the point
$\theta^{m-j}(h)\cdot g$ in $(S,\tau)$ there exist open
neighbourhoods $V_h$ and $W_g$ of the points $h$ and $g$ in
$(S,\tau)$, respectively, such that
\begin{equation*}
    \theta^{m-j}(h)\cdot W_g\subseteq U_{\theta^{m-j}(h)\cdot g}
    \qquad \hbox{and} \qquad
    \theta^{m-j}(V_h)\cdot g\subseteq U_{\theta^{m-j}(h)\cdot g}.
\end{equation*}
Hence for every $I_\alpha\in\mathscr{I}$ we have that
\begin{equation*}
\begin{split}
 &(i,h,j)\, \cdot \left(W_g\right)^\alpha_{m,n}
   \subseteq (i,h,j)\cdot\big(\left(W_g\right)_{m,n}\cup
   \left(\theta^{-1}(W_g)\cap I_\alpha\right)_{m-1,n-1}\big)\subseteq\\
   \subseteq & \big((i,h,j)\cdot\left(W_g\right)_{m,n}\big)\cup
   \big((i,h,j)\cdot\left(\theta^{-1}(W_g)\cap
   I_\alpha\right)_{m-1,n-1}\big)
   \subseteq\\
   &\!\!\!\!\!
\left\{\!\!\!
  \begin{array}{ll}
    \!\left(\theta^{m-j}(h){\cdot} W_g\right)_{i{-}j{+}m,n}\!{\cup}\!
    \left(\theta^{m{-}1{-}j}(h){\cdot}\!
    \left(\theta^{{-}1}(W_{g}){\cap} I_\alpha\right)\right)_{i{-}j{+}m{-}1,n{-}1}{,}
     & \!\!\!\!\hbox{if~} j{<}m{-}1;\\
    \left(\theta(h)\cdot W_g\right)_{i-j+m,n}\cup
    \left(h\cdot
    \left(\theta^{-1}(W_g)\cap I_\alpha\right)\right)_{i,n-1},
     & \!\!\!\!\hbox{if~} j{=}m{-}1
  \end{array}
\right. \\
  \subseteq & \left(U_{\theta^{m-j}(h)\cdot
  g}\right)_{i-j+m,n}^\alpha,
\end{split}
\end{equation*}
because
$\theta\left(\theta^{m-1-j}(h)\cdot\left(\theta^{-1}(W_g)\cap
I_\alpha\right)\right)\subseteq\theta^{m-j}(h)\cdot W_g\subseteq
U_{\theta^{m-j}(h)\cdot g}$, and
\begin{equation*}
\begin{split}
  & \left(V_h\right)^\alpha_{i,j} \cdot (m,g,n)\subseteq
    \big(\left(V_h\right)_{i,j}\cup\left(\theta^{-1}(V_h)\cap
    I_\alpha\right)_{i-1,j-1}\big)\cdot (m,g,n)\subseteq \\
  \subseteq &
  \big(\left(V_h\right)_{i,j}\cdot (m,g,n)\big)\cup
  \big(\left(\theta^{-1}(V_h)\cap I_\alpha\right)_{i-1,j-1}\cdot
  (m,g,n)\big)\subseteq \\
    \subseteq &
  \left(\theta^{m-j}(V_h)\cdot g\right)_{i-j+m,n}\cup
  \left(\theta^{m-j+1}\left(\theta^{-1}(V_h)\cap I_\alpha\right)
  \cdot g\right)_{i-j+m,n}\subseteq \\
    \subseteq &
  \left(\theta^{m-j}(V_h)\cdot g\right)_{i-j+m,n}\cup
  \left(\theta^{m-j}(V_h)\cdot g\right)_{i-j+m,n}\subseteq \\
    \subseteq &
  \left(\theta^{m-j}(V_h)\cdot g\right)_{i-j+m,n}\subseteq
    \left(U_{\theta^{m-j}(h)\cdot
  g}\right)_{i-j+m,n}\subseteq
  \left(U_{\theta^{m-j}(h)\cdot
  g}\right)_{i-j+m,n}^\alpha.
\end{split}
\end{equation*}

Suppose that $j=m$. Then the separate continuity of the semigroup
operation in $(S,\tau)$ implies that for every open neighbourhood
$U_{h\cdot g}$ of the point $h\cdot g$ in $(S,\tau)$ there exist
open neighbourhoods $V_h$ and $W_g$ of the points $h$ and $g$ in
$(S,\tau)$, respectively, such that
\begin{equation*}
    V_h\cdot g\subseteq U_{h\cdot g} \qquad \hbox{and} \qquad
    h\cdot W_g\subseteq U_{h\cdot g}.
\end{equation*}
Then for every $I_\alpha\in\mathscr{I}$ we have that
\begin{equation*}
\begin{split}
  \left(V_h\right)^\alpha_{i,j}&{\cdot}(m,g,n)
  {\subseteq}
   \big(\left(V_h\right)_{i,j}{\cdot}(m,g,n)\big)\cup
   \big(\left(\theta^{-1}(V_h)\cap
   I_\alpha\right)_{i-1,j-1}{\cdot}(m,g,n)\big)\subseteq \\
  \subseteq &
   \left(V_h\cdot g\right)_{i,n}\cup
   \left(\theta\left(\theta^{-1}(V_h)\cap
   I_\alpha\right)\cdot g\right)_{i,n}\subseteq
   \left(V_h\cdot g\right)_{i,n}\cup
   \left(V_h\cdot g\right)_{i,n}= \\
  = &
   \left(V_h\cdot g\right)_{i,n} \subseteq
   \left(U_{h\cdot g}\right)^\alpha_{i,n},
\end{split}
\end{equation*}
and
\begin{equation*}
\begin{split}
  (i,h,j)&{\cdot}\left(W_g\right)^\alpha_{m,n}
  {\subseteq}
   \big((i,h,j){\cdot}\left(W_g\right)_{m,n}\big){\cup}
   \big((i,h,j){\cdot}\left(\theta^{-1}(W_g){\cap}
   I_\alpha\right)_{m-1,n-1}\big)\subseteq \\
  \subseteq &
   \left(h\cdot W_g\right)_{i,n}\cup
   \left(h\cdot\theta\left(\theta^{-1}(W_g)\cap
   I_\alpha\right)\right)_{i,n}\subseteq
   \left(h\cdot W_g\right)_{i,n}\cup
   \left(h\cdot W_g\right)_{i,n}=\\
   =&\left(h\cdot
   W_g\right)_{i,n} \subseteq
   \left(U_{h\cdot g}\right)_{i,n}^\alpha.
\end{split}
\end{equation*}

Suppose that $j>m$. Then the separate continuity of the semigroup
operation in $(S,\tau)$ and the continuity of the homomorphism
$\theta\colon S\rightarrow H(1_S)$ imply that for every open
neighbourhood $U_{h\cdot \theta^{j-m}(g)}$ of the point $h\cdot
\theta^{j-m}(g)$ in $(S,\tau)$ there exist open neighbourhoods $V_h$
and $W_g$ of the points $h$ and $g$ in $(S,\tau)$, respectively,
such that
\begin{equation*}
    h\cdot \theta^{j-m}(W_g)\subseteq U_{h\cdot\theta^{j-m}(g)}
    \qquad \hbox{and} \qquad
    V_h\cdot \theta^{j-m}(g)\subseteq U_{h\cdot\theta^{j-m}(g)}.
\end{equation*}
Hence for every $I_\alpha\in\mathscr{I}$ we have that
\begin{equation*}
\begin{split}
  (i,h,j) &{\cdot}\left(W_g\right)^\alpha_{m,n}
  {\subseteq}
   \big((i,h,j){\cdot}\left(W_g\right)_{m,n}\big){\cup}
   \big((i,h,j){\cdot}\left(\theta^{-1}(W_g){\cap}
   I_\alpha\right)_{m-1,n-1}\big){\subseteq} \\
  \subseteq &
   \left(h\cdot\theta^{j-m}(W_g)\right)_{i,n-m+j}\cup
   \left(h\cdot\theta^{j-m+1}\left(\theta^{-1}(W_g)\cap
   I_\alpha\right)\right)_{i,n-m+j}\subseteq \\
  \subseteq &
   \left(h\cdot\theta^{j-m}(W_g)\right)_{i,n-m+j}\cup
   \left(h\cdot\theta^{j-m}(W_g)\right)_{i,n-m+j}=\\
   =&\left(h\cdot\theta^{j-m}(W_g)\right)_{i,n-m+j}
  \subseteq
   \left(U_{h\cdot\theta^{j-m}(g)}\right)^\alpha_{i,n-m+j},
\end{split}
\end{equation*}
and
\begin{equation*}
\begin{split}
  \!\!&\left(V_h\right)^\alpha_{i,j} {\cdot}(m,g,n)
  {\subseteq}
   \big(\left(V_h\right)_{i,j}{\cdot}(m,g,n)\big){\cup}
   \big(\left(\theta^{-1}(V_h){\cap}
   I_\alpha\right)_{i{-}1,j{-}1}{\cdot}(m,g,n)\big){\subseteq} \\
  &\!\!\!\!\!\!
\left\{\!\!\!
  \begin{array}{ll}
    \left(V_h{\cdot}\theta^{j{-}m}(g)\right)_{i,n{-}m{+}j}\!{\cup}
    \left(\left(\theta^{-1}(V_h)\cap I_\alpha\right)
    {\cdot} g\right)_{i{-}1,n},
   & \!\!\hbox{if~} j{-}1{=}m;\\
    \left(V_h{\cdot}\theta^{j{-}m}(g)\right)_{i,n{-}m{+}j}\!{\cup}
    \left(\left(\theta^{-1}(V_h){\cap} I_\alpha\right)\!
    {\cdot} \theta^{j{-}1{-}m}(g)\right)_{i{-}1,n{-}m{+}j{-}1}{,}
   & \!\!\hbox{if~} j{-}1{>}m
  \end{array}
\right.\\
\subseteq &
   \left(U_{h\cdot\theta^{j-m}(g)}\right)^\alpha_{i,n-m+j},
\end{split}
\end{equation*}
because $\theta\left(\left(\theta^{-1}(V_h)\cap
I_\alpha\right)\cdot\theta^{j-1-m}(g)\right)=V_h\cdot\theta^{j-m}(g)
\subseteq U_{h\cdot\theta^{j-m}(g)}$.

We observe that if $g\in H(1_S)$ and $s\in S\setminus H(1_S)$ then
$g\cdot s, s\cdot g\in S\setminus H(1_S)$. Otherwise, if $g\cdot
s\in H(1_S)$ then we have that $g^{-1}\cdot g\cdot s=1_S\cdot s=s\in
H(1_S)$, which contradicts that every translation on an element of
the group of units of $S$ is a bijective map (see \cite[Vol.~1,
p.~18]{CHK}).

Consider case $2)$. Then we have that
\begin{equation*}\label{f8}
 (i,h,j)\cdot(m,s,n)=
 \left\{
   \begin{array}{ll}
     (i-j+m,\theta^{m-j}(h)\cdot s,n), & \hbox{if~} j<m; \\
     (i,h\cdot s,n), & \hbox{if~} j=m; \\
     (i,h\cdot\theta^{j-m}(s),n-m+j), & \hbox{if~} j>m.
   \end{array}
 \right.
\end{equation*}

Suppose that $j<m$. Then the separate continuity of the semigroup
operation in $(S,\tau)$ and the continuity of the homomorphism
$\theta\colon S\rightarrow H(1_S)$ imply that for every open
neighbourhood $U_{\theta^{m-j}(h)\cdot s}$ of the point
$\theta^{m-j}(h)\cdot s$ in $(S,\tau)$ there exist open
neighbourhoods $V_h$ and $W_s$ of the points $h$ and $g$ in
$(S,\tau)$, respectively, such that
\begin{equation*}
    \theta^{m-j}(h)\cdot W_s\subseteq U_{\theta^{m-j}(h)\cdot s}
    \qquad \hbox{and} \qquad
    \theta^{m-j}(V_h)\cdot s\subseteq U_{\theta^{m-j}(h)\cdot s}.
\end{equation*}
Hence for every $I_\alpha\in\mathscr{I}$ we have that
\begin{equation*}
(i,h,j)\cdot\left(W_s\right)_{m,n}\subseteq
\left(\theta^{m-j}(h)\cdot W_s\right)_{i-j+m,n}\subseteq \left(U_{\theta^{m-j}(h)\cdot
s}\right)_{i-j+m,n}
\end{equation*}
 and
\begin{equation*}
\begin{split}
  \left(V_h\right)_{i,j}^\alpha&{\cdot}(m,s,n){\subseteq}
   \big(\left(V_h\right)_{i,j}{\cdot}(m,s,n)\big){\cup}
   \big(\left(\theta^{-1}(V_h){\cap} I_\alpha\right)_{i-1,j-1}
   {\cdot}(m,s,n)\big)
   {\subseteq} \\
  \subseteq &
   \left(\theta^{m-j}(V_h)\cdot s\right)_{i-j+m,n}\cup
   \left(\theta^{m-j+1}\left(\theta^{-1}(V_h)\cap
   I_\alpha\right)\cdot s\right)_{i-j+m,n}\subseteq \\
  \subseteq &
   \left(\theta^{m-j}(V_h)\cdot s\right)_{i-j+m,n}\cup
   \left(\theta^{m-j}(V_h)\cdot s\right)_{i-j+m,n}\subseteq \\
  \subseteq &
   \left(\theta^{m-j}(V_h)\cdot s\right)_{i-j+m,n}\subseteq
   \left(U_{\theta^{m-j}(h)\cdot s}\right)_{i-j+m,n}.
\end{split}
\end{equation*}

Suppose that $j=m$. Then the separate continuity of the semigroup
operation in $(S,\tau)$ implies that for every open neighbourhood
$U_{h\cdot s}$ of the point $h\cdot s$ in $(S,\tau)$ there exist
open neighbourhoods $V_h$ and $W_s$ of the points $h$ and $s$ in
$(S,\tau)$, respectively, such that
\begin{equation*}
    V_h\cdot s\subseteq U_{h\cdot s} \qquad \hbox{and} \qquad
    h\cdot W_s\subseteq U_{h\cdot s}.
\end{equation*}
Then for every $I_\alpha\in\mathscr{I}$ we have that
$(i,h,j)\cdot\left(W_s\right)_{m,n}\subseteq \left(h\cdot
W_s\right)_{i,n}\subseteq\left(U_{h\cdot s}\right)_{i,n}$ and
\begin{equation*}
\begin{split}
  \left(V_h\right)^\alpha_{i,j}&{\cdot}(m,s,n) {\subseteq}
   \big(\left(V_h\right)_{i,j}{\cdot}(m,s,n)\big){\cup}
   \big(\left(\theta^{-1}(V_h){\cap}
   I_\alpha\right)_{i-1,j-1}{\cdot}(m,s,n)\big)\subseteq\\
  \subseteq &
   \left(V_h\cdot s\right)_{i,n}\cup
   \left(\theta\left(\theta^{-1}(V_h)\cap
   I_\alpha\right)\cdot s\right)_{i,n}
  \subseteq
   \left(V_h\cdot s\right)_{i,n}\cup
   \left(V_h\cdot s\right)_{i,n}=\\
  = &
   \left(V_h\cdot s\right)_{i,n}\subseteq
   \left(U_{h\cdot s}\right)_{i,n}.
\end{split}
\end{equation*}

If $j>m$  then the separate continuity of the semigroup operation in
$(S,\tau)$ and the continuity of the homomorphism $\theta\colon
S\rightarrow H(1_S)$ imply that for every open neighbourhood
$U_{h\cdot \theta^{j-m}(s)}$ of the point $h\cdot \theta^{j-m}(s)$
in $(S,\tau)$ there exist open neighbourhoods $V_h$ and $W_s$ of the
points $h$ and $s$ in $(S,\tau)$, respectively, such that
\begin{equation*}
    h\cdot \theta^{j-m}(W_s)\subseteq U_{h\cdot\theta^{j-m}(s)}
    \qquad \hbox{and} \qquad
    V_h\cdot \theta^{j-m}(s)\subseteq U_{h\cdot\theta^{j-m}(s)}.
\end{equation*}
Hence for every $I_\alpha\in\mathscr{I}$ we have that
\begin{equation*}
(i,h,j)\cdot\left(W_s\right)_{m,n}\subseteq \left(h\cdot
\theta^{j-m}(W_s)\right)_{i,n-m+j}\subseteq
\left(U_{h\cdot\theta^{j-m}(s)}\right)_{i,n-m+j}^\alpha
\end{equation*}
 and
\begin{equation*}
\begin{split}
  &\left(V_h\right)^\alpha_{i,j}{\cdot}(m,s,n){\subseteq}
   \big(\left(V_h\right)_{i,j}{\cdot} (m,s,n)\big){\cup}
   \big(\left(\theta^{-1}(V_h){\cap}
   I_\alpha\right)_{i-1,j-1}{\cdot} (m,s,n)\big){\subseteq}\\
  &
\left\{\!\!\!\!
  \begin{array}{ll}
    \left(V_h{\cdot} \theta^{j{-}m}(s)\right)_{i,n{-}m{+}j}{\cup}
   \left(\left(\theta^{-1}(V_h){\cap}
   I_\alpha\right)\!{\cdot} s\right)_{i{-}1,n}{,}
    & \!\!\!\hbox{if~} j{-}1{=}m;\\
    \left(V_h{\cdot} \theta^{j-m}(s)\right)_{i,n{-}m{+}j}{\cup}
   \left(\left(\theta^{-1}(V_h){\cap}
   I_\alpha\right)\!{\cdot} \theta^{j{-}1{-}m}(s)\right)_{i{-}1,n{-}m{+}j{-}1}{,}
    & \!\!\!\hbox{if~} j{-}1{>}m
  \end{array}
\right.
  \\
   &\subseteq
    \left(V_h\cdot \theta^{j-m}(s)\right)^\alpha_{i,n-m+j}\subseteq
    \left(U_{h\cdot\theta^{j-m}(s)}\right)^\alpha_{i,n-m+j},
\end{split}
\end{equation*}
because $\theta\left(\left(\theta^{-1}(V_h)\cap I_\alpha\right)\cdot
\theta^{j-1-m}(s)\right)\subseteq V_h\cdot \theta^{j-m}(s)\subseteq
U_{h\cdot\theta^{j-m}(s)}$.

In case $3)$ we have that
\begin{equation*}\label{f9}
 (m,s,n)\cdot(i,g,j)=
 \left\{
   \begin{array}{ll}
     (m-n+i,\theta^{i-n}(s)\cdot g,j), & \hbox{if~} n<i; \\
     (m,s\cdot g,j), & \hbox{if~} n=i; \\
     (m,s\cdot\theta^{n-i}(g),j-i+n), & \hbox{if~} n>i.
   \end{array}
 \right.
\end{equation*}
and in this case the proof of separate continuity of the semigroup
operation in $(\mathscr{B}(S,\mathbb{Z},\theta),\tau_{\textbf{BR}})$
is similar to case $2)$.
\end{proof}

We observe that in the case when $\theta(s)=1_S$ for all $s\in S^1$
then a base of the topology $\tau_{\textbf{BR}}$ on
$\mathscr{B}(S,\mathbb{Z})$ is determined in the following way:
\begin{itemize}
 \item[$(1)$] for every $s\in S^1\setminus\{1_S\}$ and
  $i,j\in\mathbb{Z}$ the base of the topology
  $\tau_{\textbf{BR}}$ at the point $(i,s,j)$ coincides with a
  base of the direct sum topology $\tau_{\textbf{BR}}^\textbf{ds}$
  at $(i,s,j)$; \; and

 \item[$(2)$] the family $\mathscr{B}_{(i,1_S,j)}=\big\{
  U^\alpha_{i,j}=U_{i,j}\cup(I_\alpha)_{i-1,j-1}\colon U\in
  \mathscr{B}_{1_S},I_\alpha\in\mathscr{I}\big\}$, where
  $\mathscr{B}_{1_S}$ is a base of the topology $\tau$ at the point
  $1_S$ in $S$, is a base of the topology $\tau_{\textbf{BR}}$ at
  the point $(i,1_S,j)$, for all $i,j\in\mathbb{Z}$.
\end{itemize}

Then Theorem~\ref{theorem-2.10} implies the following:

\begin{theorem}\label{theorem-2.10a}
Let $(S,\tau)$ be a Hausdorff semitopological
\textsf{OIP}-semigroup. Then there exists a topological
$\mathbb{Z}$-Bruck extension $\left(\mathscr{B}(S,\mathbb{Z}),
\tau_{\textbf{BR}}\right)$ of $(S,\tau)$ in the class of
semitopological semigroups such that the topology
$\tau_{\textbf{BR}}$ is strictly coarser than the direct sum topology
$\tau_{\textbf{BR}}^\textbf{ds}$ on $\mathscr{B}(S,\mathbb{Z})$.
\end{theorem}

Later we need the following:

\begin{proposition}\label{proposition-2.11}
Let $(S,\tau)$ be a topological (inverse) \textsf{OIP}-semigroup.
Let $\tau_{\textbf{BR}}$ be a topology on the semigroup
$\mathscr{B}(S,\mathbb{Z},\theta)$ which is determined in the proof
of Theorem~\ref{theorem-2.10}. Then
$\left(\mathscr{B}(S,\mathbb{Z},\theta), \tau_{\textbf{BR}}\right)$
is a topological (inverse) semigroup.
\end{proposition}

\begin{proof}
If $(S,\tau)$ is a topological semigroup then
Proposition~\ref{proposition-2.1} implies that the semigroup
operation is continuous in $\left(\mathscr{B}(S,\mathbb{Z},\theta),
\tau_{\textbf{BR}}^\textbf{ds}\right)$. Similarly, if inversion in
an inverse topological semigroup $(S,\tau)$ is continuous then
Proposition~\ref{proposition-2.1} implies that the inversion in
$\left(\mathscr{B}(S,\mathbb{Z},\theta),
\tau_{\textbf{BR}}^\textbf{ds}\right)$ is continuous too. Therefore
it is sufficient to show that the semigroup operation is jointly
continuous in $\left(\mathscr{B}(S,\mathbb{Z},\theta),
\tau_{\textbf{BR}}\right)$ in the following three cases:
\begin{equation*}
1)~(i,h,j)\cdot(m,g,n); \;
2)~(i,h,j)\cdot(m,s,n); \; \hbox{~and~} \;
3)~(m,s,n)\cdot(i,g,j),
\end{equation*}
and in the case when $(S,\tau)$ is an topological inverse semigroup
it is sufficient to show that inversion is continuous at the point
$(i,h,j)$, for all $h,g\in H(1_S)$, $s\in S\setminus H(1_S)$ and
$i,j,m,n\in\mathbb{Z}$.

Consider case $1)$. Then we have that
\begin{equation*}\label{f10}
 (i,h,j)\cdot(m,g,n)=
 \left\{
   \begin{array}{ll}
     (i-j+m,\theta^{m-j}(h)\cdot g,n), & \hbox{if~} j<m; \\
     (i,h\cdot g,n), & \hbox{if~} j=m; \\
     (i,h\cdot\theta^{j-m}(g),n-m+j), & \hbox{if~} j>m.
   \end{array}
 \right.
\end{equation*}

If $j<m$ then the continuity of the semigroup operation in
$(S,\tau)$ and the continuity of the homomorphism $\theta\colon
S\rightarrow H(1_S)$ imply that for every open neighbourhood
$U_{\theta^{m-j}(h)\cdot g}$ of the point $\theta^{m-j}(h)\cdot g$
in $(S,\tau)$ there exist open neighbourhoods $V_h$ and $W_g$ of the
points $h$ and $g$ in $(S,\tau)$, respectively, such that
$\theta^{m-j}(V_h)\cdot W_g\subseteq U_{\theta^{m-j}(h)\cdot g}$.
Hence for every $I_\alpha\in \mathscr{I}$ we have that
\begin{equation*}
\begin{split}
  &\left(V_h\right)^\alpha_{i,j}{\cdot}
  \left(W_g\right)^\alpha_{m,n}  {\subseteq}
  \big(\left(V_h\right)_{i,j}{\cdot}\left(W_g\right)_{m,n}\big){\cup}
  \big(\left(V_h\right)_{i,j}{\cdot}
  \left(\theta^{-1}(W_g){\cap} I_\alpha\right)_{m-1,n-1}\big){\cup}\\
     &
  \big(\!\left(\theta^{-1}(V_h){\cap} I_\alpha\right)_{i{-}1,j{-}1}
  \!\!{\cdot}\!\left(W_g\right)_{m,n}\!\big){\cup}
  \big(\!\left(\theta^{-1}(V_h){\cap} I_\alpha\right)_{i{-}1,j{-}1}\!\! {\cdot}\!
  \left(\theta^{-1}(W_g){\cap} I_\alpha\right)_{m{-}1,n{-}1}\!\big)\\
    & \subseteq
  \left(\theta^{m-j}(V_h)\cdot W_g\right)_{i-j+m,n}\cup A\cup
  \left(\theta^{m-j+1}\left(\theta^{-1}(V_h)\cap I_\alpha\right)
  \cdot W_g\right)_{i-j+m,n}\cup\\
     &{\cup}
  \left(\theta^{m-j}\left(\theta^{-1}(V_h){\cap} I_\alpha\right)
  {\cdot} \left(\theta^{-1}(W_g){\cap}
  I_\alpha\right)\right)_{i-j+m-1,n-1}\subseteq \left(U_{\theta^{m-j}(h)\cdot
     g}\right)_{i-j+m,n}^\alpha,
\end{split}
\end{equation*}
where
\begin{equation*}
A=
\left\{
  \begin{array}{ll}
    \left(V_h\cdot\left(\theta^{-1}(W_g)\cap
    I_\alpha\right)\right)_{i,n-1}, & \hbox{if~} j=m-1;\\
    \left(\theta^{m-1-j}(V_h)\cdot\left(\theta^{-1}(W_g)\cap
    I_\alpha\right)\right)_{i-j+m-1,n-1}, & \hbox{if~} j<m-1,
  \end{array}
\right.
\end{equation*}
because
\begin{equation*}
\begin{split}
& \theta^{m-j+1}\left(\theta^{-1}(V_h)\cap I_\alpha\right)
  \cdot W_g\subseteq \theta^{m-j}(V_h)\cdot W_g\subseteq
  U_{\theta^{m-j}(h)\cdot g},\\
& \theta\left(\theta^{m-j}\left(\theta^{-1}(V_h)\cap I_\alpha\right)
  \cdot \left(\theta^{-1}(W_g)\cap
  I_\alpha\right)\right)\subseteq \theta^{m-j}(V_h)\cdot W_g\subseteq
  U_{\theta^{m-j}(h)\cdot g}, \\
\end{split}
\end{equation*}
and
\begin{equation*}
\theta(A)= \left\{
  \begin{array}{ll}
    \theta(V_h)\cdot W_g, & \hbox{if~} j=m-1;\\
    \theta^{m-j}(V_h)\cdot W_g, & \hbox{if~} j<m-1,
  \end{array}
\right. \subseteq U_{\theta^{m-j}(h)\cdot g}.
\end{equation*}

The proof of the continuity of the semigroup operation in
$\left(\mathscr{B}(S,\mathbb{Z},\theta), \tau_{\textbf{BR}}\right)$
in the case when $j>m$ is similar to the previous case.

If $j=m$ then the continuity of the semigroup operation in
$(S,\tau)$ implies that for every open neighbourhood $U_{h\cdot g}$
of the point $h\cdot g$ in $(S,\tau)$ there exist open
neighbourhoods $V_h$ and $W_g$ of the points $h$ and $g$ in
$(S,\tau)$, respectively, such that $V_h\cdot W_g\subseteq U_{h\cdot
g}$. Then for every $I_\alpha\in \mathscr{I}$ we get that
\begin{equation*}
\left(V_h\right)^\alpha_{i,j}\cdot\left(W_g\right)^\alpha_{m,n}
\subseteq \left(V_h\cdot W_g\right)^\alpha_{i,n} \subseteq
\left(U_{h\cdot g}\right)^\alpha_{i,n}.
\end{equation*}

In case $2)$  we have that
\begin{equation*}\label{f11}
 (i,h,j)\cdot(m,s,n)=
 \left\{
   \begin{array}{ll}
     (i-j+m,\theta^{m-j}(h)\cdot s,n), & \hbox{if~} j<m; \\
     (i,h\cdot s,n), & \hbox{if~} j=m; \\
     (i,h\cdot\theta^{j-m}(s),n-m+j), & \hbox{if~} j>m,
   \end{array}
 \right.
\end{equation*}
where $\theta^{m-j}(h)\cdot s, h\cdot s\in S\setminus H(1_S)$ and
$h\cdot \theta^{j-m}(s)\in H(1_S)$.

If $j<m$ then the continuity of the semigroup operation in
$(S,\tau)$ and the continuity of the homomorphism $\theta\colon
S\rightarrow H(1_S)$ imply that for every open neighbourhood
$U_{\theta^{m-j}(h)\cdot s}$ of the point $\theta^{m-j}(h)\cdot s$
in $(S,\tau)$ there exist open neighbourhoods $V_h$ and $W_s$ of the
points $h$ and $s$ in $(S,\tau)$, respectively, such that
$\theta^{m-j}(V_h)\cdot W_s\subseteq U_{\theta^{m-j}(h)\cdot s}$.
Hence for every $I_\alpha\in \mathscr{I}$ we have that
\begin{equation*}
\begin{split}
  \left(V_h\right)^\alpha_{i,j}& {\cdot}
  \left(W_s\right)_{m,n}
 \subseteq
  \big(\left(V_h\right)_{i,j}{\cdot}\left(W_s\right)_{m,n}\big)\cup
  \big(\left(\theta^{-1}(V_h)\cap I_\alpha\right)_{i-1,j-1}
  {\cdot}\left(W_s\right)_{m,n}\big)\subseteq\\
 \subseteq &
  \left(\theta^{m-j}(V_h)\cdot W_s\right)_{i-j+m,n}\cup
  \left(\theta^{m-j+1}\left(\theta^{-1}(V_h)\cap I_\alpha\right)
  \cdot W_s\right)_{i-j+m,n}\subseteq\\
 \subseteq &
  \left(\theta^{m-j}(V_h)\cdot W_s\right)_{i-j+m,n}\subseteq
  \left(U_{\theta^{m-j}(h)\cdot s}\right)_{i-j+m,n}.
\end{split}
\end{equation*}

If $j=m$ then the continuity of the semigroup operation in
$(S,\tau)$ implies that for every open neighbourhood $U_{h\cdot s}$
of the point $h\cdot s$ in $(S,\tau)$ there exist open
neighbourhoods $V_h$ and $W_s$ of the points $h$ and $s$ in
$(S,\tau)$, respectively, such that $V_h\cdot W_s\subseteq U_{h\cdot
s}$. Then for every $I_\alpha\in \mathscr{I}$ we get that
\begin{equation*}
\begin{split}
\left(V_h\right)^\alpha_{i,j}&{\cdot}\left(W_s\right)_{m,n}{\subseteq}
  \big(\left(V_h\right)_{i,j}{\cdot}\left(W_s\right)_{m,n}\big){\cup}
  \big(\left(\theta^{-1}\left(V_h\right){\cap}
  I_\alpha\right)_{i-1,j-1}{\cdot}\left(W_s\right)_{m,n}\big)\subseteq\\
\subseteq&
  \left(V_h\cdot W_s\right)_{i,n}\cup
  \left(\theta\left(\theta^{-1}\left(V_h\right)\cap
  I_\alpha\right)\cdot W_s\right)_{i,n}\subseteq
  \left(V_h\cdot W_s\right)_{i,n}\subseteq\left(U_{h\cdot
  s}\right)_{i,n}.
\end{split}
\end{equation*}

If $j>m$ then the continuity of the semigroup operation in
$(S,\tau)$ and the continuity of the homomorphism $\theta\colon
S\rightarrow H(1_S)$ imply that for every open neighbourhood
$U_{h\cdot\theta^{j-m}(s)}$ of the point $h\cdot\theta^{j-m}(s)$ in
$(S,\tau)$ there exist open neighbourhoods $V_h$ and $W_s$ of the
points $h$ and $s$ in $(S,\tau)$, respectively, such that $V_h\cdot
\theta^{j-m}(W_s)\subseteq U_{h\cdot\theta^{j-m}(s)}$. Hence for
every $I_\alpha\in \mathscr{I}$ we have that
\begin{equation*}
\begin{split}
  \!&\left(V_h\right)^\alpha_{i,j}{\cdot}\left(W_s\right)_{m,n} {\subseteq}
  \big(\left(V_h\right)_{i,j}{\cdot}\left(W_s\right)_{m,n}\big){\cup}
  \big(\left(\theta^{-1}\left(V_h\right){\cap}
  I_\alpha\right)_{i-1,j-1}{\cdot}\left(W_s\right)_{m,n}\big)\subseteq\\
 & \!
 \left\{\!\!\!\!
  \begin{array}{ll}
    \left(V_h{\cdot} \theta^{j-m}(W_s)\!\right)_{i,n-m+j}\!{\cup}\!
  \left(\!\left(\theta^{-1}\!\left(V_h\right)\!{\cap}
  I_\alpha\right)\!{\cdot} W_s\right)_{i-1,n},
   & \!\!\!\!\hbox{if~} j{-}1{=}m;\\
    \left(V_h{\cdot} \theta^{j-m}(W_s)\!\right)_{i,n-m+j}\!{\cup}\!
  \left(\!\left(\theta^{-1}\!\left(V_h\right)\!{\cap}
  I_\alpha\right)\!{\cdot}\theta^{j-1-m}(W_s)\!\right)_{i-1,n-m+j-1}\!\!{,}
   & \!\!\!\!\hbox{if~} j{-}1{>}m
  \end{array}
\right. \\
 &\subseteq
  \left(V_h\cdot \theta^{j-m}(W_s)\right)_{i,n-m+j}\cup
  \left(\theta^{-1}\left(U_{h\cdot\theta^{j-m}(s)}\right)\cap
  I_\alpha\right)_{i-1,n-m+j-1}\subseteq\\
 &\subseteq
  \left(U_{h\cdot\theta^{j-m}(s)}\right)^\alpha_{i,n},
\end{split}
\end{equation*}
because
\begin{equation*}
\theta\left(\left(\theta^{-1}\left(V_h\right)\cap
I_\alpha\right)\cdot\theta^{j-1-m}(W_s)\right)\subseteq V_h\cdot
\theta^{j-m}(W_s)\subseteq U_{h\cdot\theta^{j-m}(s)}.
\end{equation*}

The proof of the continuity of the semigroup operation in
$\left(\mathscr{B}(S,\mathbb{Z},\theta), \tau_{\textbf{BR}}\right)$
in case $3)$ is similar to case $2)$.

If $(S,\tau)$ is a topological inverse semigroup then for every
ideal $I$ in $S$ we have that $I^{-1}=I$ and for every open
neighbourhoods $V_s$ and $U_{s^{-1}}$ of the points $s$ and $s^{-1}$
in $(S,\tau)$, respectively, such that
$\left(V_s\right)^{-1}\subseteq U_{s^{-1}}$ we have that
\begin{itemize}
  \item[] $\big(\left(V_s\right)_{i,j}\big)^{-1}\subseteq
   \big(U_{s^{-1}}\big)_{j,i}$, for $s\in S\setminus H(1_S)$ \;
    and
  \item[] $\big(\left(V_s\right)^\alpha_{i,j}\big)^{-1}\subseteq
   \left(U_{s^{-1}}\right)^\alpha_{j,i}$, for $s\in H(1_S)$,
\end{itemize}
for all $I_\alpha\in\mathscr{I}$, and hence
$\left(\mathscr{B}(S,\mathbb{Z},\theta), \tau_{\textbf{BR}}\right)$
is a topological inverse semigroup. This completes the proof of the
proposition.
\end{proof}

Theorem~\ref{theorem-2.10} and Proposition~\ref{proposition-2.11}
imply the following:

\begin{theorem}\label{theorem-2.12}
Let $(S,\tau)$ be a topological (inverse) \textsf{OIP}-semigroup.
Then there exists a topological $\mathbb{Z}$-Bruck-Reilly extension
$\left(\mathscr{B}(S,\mathbb{Z},\theta), \tau_{\textbf{BR}}\right)$
of $(S,\tau)$ in the class of topological (inverse) semigroups such
that the topology $\tau_{\textbf{BR}}$ is strictly coarser than the
direct sum topology $\tau_{\textbf{BR}}^\textbf{ds}$ on
$\mathscr{B}(S,\mathbb{Z},\theta)$.
\end{theorem}

Theorem~\ref{theorem-2.12} implies the following:

\begin{corollary}\label{corollary-2.12a}
Let $(S,\tau)$ be a topological (inverse) \textsf{OIP}-semigroup.
Then there exists a topological $\mathbb{Z}$-Bruck extension
$\left(\mathscr{B}(S,\mathbb{Z}), \tau_{\textbf{BR}}\right)$ of
$(S,\tau)$ in the class of topological (inverse) semigroups such
that the topology $\tau_{\textbf{BR}}$ is strictly coarser than the
direct sum topology $\tau_{\textbf{BR}}^\textbf{ds}$ on
$\mathscr{B}(S,\mathbb{Z})$.
\end{corollary}

Recall \cite{CHK}, a topological semilattice $E$ is said to be a
\emph{$U$-semilattice} if for every $x\in E$ and every open
neighbourhood $U={\uparrow}U$ of $x$ in $E$, there exists $y\in U$
such that $x\in\operatorname{Int}_E({\uparrow}y)$.

\begin{remark}\label{remark-2.13}
Let $S$ be a Clifford inverse semigroup. We define a map
$\varphi\colon S\rightarrow E(S)$ by the formula $\varphi(x)=x\cdot
x^{-1}$. Theorem~4.11 from \cite{CP} implies that if $I$ is an ideal
of $E(S)$ then $\varphi^{-1}(I)$ is an ideal of $S$.
\end{remark}

The following theorem gives examples of topological
\textsf{OIP}-semigroups.

\begin{theorem}\label{theorem-2.14}
Let $(S,\tau)$ be a topological inverse Clifford semigroup. If the
band $E(S)$ of $S$ has no a smallest idempotent and satisfies one of
the following conditions:
\begin{itemize}
  \item[$(1)$] for every $x\in E(S)$ there exists $y\in{\downarrow}x$
   such that there is an open neighbourhood $U_y$ of $y$ with the compact
   closure $\operatorname{cl}_{E(S)}(U_y)$;
  \item[$(2)$] $E(S)$ is locally compact;
  \item[$(3)$] $E(S)$ is a $U$-semilattice,
\end{itemize}
then $(S,\tau)$ is an \textsf{OIP}-semigroup.
\end{theorem}

\begin{proof}
Suppose condition $(1)$ holds. We fix an arbitrary $x\in E(S)$. By Proposition~VI-1.14 of \cite{GHKLMS} the partial order on the topological semilattice $E(S)$ is closed, and hence the compact set $K=\operatorname{cl}_{E(S)}(U_y)$ has a minimal element $e$, which must also be a minimal element of ${\uparrow}K$. If ${\uparrow}K=E(S)$, then $e$ is a minimal element of $E(S)$ and hence $e$ is a least element of $E(S)$, because $ef\leqslant e$ for any $f\in E(S)$ implies $e=ef$, i.e., $e\leqslant f$. This contradicts that $E(S)$ hasn't a least element.

Then the set $I_x=E(S)\setminus
{\uparrow}\left(\operatorname{cl}_{E(S)}(U_y)\right)$ is an open
ideal in $E(S)$ and by Proposition~VI-1.13$(iii)$ from \cite{GHKLMS}
the set $U_x={\uparrow}U_y$ is an open neighbourhood of the point
$x$ in $E(S)$ such that $I_x\cap U_x=\varnothing$. Therefore for
every $x\in E(S)$ we constructed an open neighbourhood $U_x$ of the
point $x$ in $E(S)$ and an open ideal $I_x$ in $E(S)$ such that
$I_x\cap U_x=\varnothing$, and hence the topological semilattice
$E(S)$ is an \textsf{OIP}-semigroup. Now we apply
Remark~\ref{remark-2.13} and get that $(S,\tau)$ is an
\textsf{OIP}-semigroup.

We observe that every locally compact semilattice satisfies
condition $(1)$.

Suppose condition $(3)$ holds. We fix an arbitrary $x\in E(S)$.
Since the semilattice $E(S)$ does not contain a minimal idempotent
we conclude that there exists an idempotent
$e\in{\downarrow}x\setminus\{x\}$. Then by Proposition~VI-1.13$(i)$
from \cite{GHKLMS} the set $U_x=E(S)\setminus{\downarrow}e$ is open
in $E(S)$ and it is obvious that $x\in U_x={\uparrow}U_x$. Let
$y_{[x,e]}\in U_x$ be such that
$x\in\operatorname{Int}_{E(S)}({\uparrow}y_{[x,e]})$. We put
$V_x=\operatorname{Int}_{E(S)}({\uparrow}y_{[x,e]})$ and
$I_{[x,e]}=E(S)\setminus{\uparrow}y_{[x,e]}$. Then $V_x$ is an open
neighbourhood of $x$ in $E(S)$ and $I_{[x,e]}$ is an open ideal in
$E(S)$. Hence similar arguments as in case $(1)$ show that
$(S,\tau)$ is an \textsf{OIP}-semigroup.
\end{proof}
%%%%%%%%%%%%%%%%%%%%%%%%%%%%%%%%%%%%%%%%%%%%%%%%%%%%%%%%%%%%%%%
%

\section{On $I$-bisimple topological inverse semigroups}

A bisimple semigroup $S$ is called an \emph{$I$-bisimple semigroup}
if and only if $E(S)$ is order isomorphic to $\mathbb{Z}$ under the
reverse of the usual order.

In \cite{Warne1968} Warne proved the following theorem:

\begin{theorem}[{\cite[Theorem~1.3]{Warne1968}}]\label{theorem-3.1}
A regular semigroup $S$ is $I$-bisimple if and only if $S$ is
isomorphic to $\mathscr{B}_W=\mathbb{Z}\times G\times\mathbb{Z}$,
where $G$ is a group, under the multiplication
\begin{equation}\label{eq-3.1}
    (a,g,b){\cdot}(c,h,d){=}
\left\{\!\!
  \begin{array}{ll}
  (a, g\cdot f^{-1}_{b-c,c}\cdot\theta^{b-c}(h)\cdot f_{b-c,d},d-c+b),
     & \hbox{if~}~b\geqslant c;\\
  (a-b+c, f^{-1}_{c-b,a}\cdot\theta^{c-b}(g)\cdot f_{c-b,b}\cdot h,d),
     & \hbox{if~}~b\leqslant c,
  \end{array}
\right.
\end{equation}
where $\theta$ is an endomorphism of $G$, $\theta^0$ denoting the
identity automorphism of $G$, and for $m\in\mathbb{N}$, $n\in
\mathbb{Z}$,
\begin{itemize}
  \item[$(1)$] $f_{0,n}=e$ is the identity of $G$; \qquad and
  \item[$(2)$] $f_{m,n}=\theta^{m-1}(u_{n+1})\cdot
   \theta^{m-2}(u_{n+2})\cdot\ldots\cdot
   \theta(u_{n+(m-1)})\cdot u_{n+m}$, where $\{u_n\colon
   n\in\mathbb{Z}\}$ is a collection of elements of $G$ with $u_n=e$
   if $n\in\mathbb{N}$.
\end{itemize}
\end{theorem}

For arbitrary $i,j\in\mathbb{Z}$ we denote
$G_{i,j}=\{(i,g,j)\in\mathscr{B}_W\colon g\in G\}$.

\begin{theorem}\label{theorem-3.2}
Let $S$ be a regular $I$-bisimple semitopological semigroup. Then
there exist a group $G$ with the identity element $e$, an
endomorphism $\theta\colon G\rightarrow G$, a collection
$\{u_n\colon n\in\mathbb{Z}\}$ of elements of $G$ with the property
$u_n=e$ if $n\in\mathbb{N}$ and a topology on the semigroup
$\mathscr{B}_W$ such that the following assertions hold:
\begin{itemize}
  \item[$(i)$] $S$ is topologically isomorphic to a semitopological
   semigroup $\mathscr{B}_W$ (not necessarily with the product
   topology);

  \item[$(ii)$] $G_{i,j}$ and $G_{k,l}$ are homeomorphic subspaces
   of $\mathscr{B}_W$ for all $i,j,k,l\in\mathbb{Z}$;

  \item[$(iii)$] $G_{i,i}$ and $G_{k,k}$ are topologically
   isomorphic semitopological subgroups of $\mathscr{B}_W$ with the
   induced topology from $\mathscr{B}_W$ for all $i,k\in\mathbb{Z}$;

  \item[$(iv)$] $\theta$ is a continuous endomorphism of the
   semitopological group $G=G_{i,i}$ with the induced from
   $\mathscr{B}_W$ topology, for an arbitrary integer $i$;

  \item[$(v)$] for every element $(i,g,j)\in\mathscr{B}_W$ there
   exists an open neighbourhood $U_{(i,g,j)}$ of the point $(i,g,j)$
   in $\mathscr{B}_W$ such that $U_{(i,g,j)}\subseteq
   \bigcup\{G_{i-k,j-k}\colon k=0,1,2,3,\ldots\}$;

  \item[$(vi)$] $E(S)$ is a discrete subspace of $S$.
\end{itemize}
\end{theorem}

\begin{proof}
The first part of the theorem and assertion $(i)$ follow from
Theorem~\ref{theorem-3.1}.

$(ii)$ We fix arbitrary $i,j,k,l\in\mathbb{Z}$ and define the map
$\varphi_{i,j}^{k,l}\colon \mathscr{B}_W\rightarrow \mathscr{B}_W$
by the formula $\varphi_{i,j}^{k,l}(x)=(k,e,i)\cdot x\cdot(j,e,l)$.
Then formula (\ref{eq-3.1}) implies that the restriction
$\varphi_{i,j}^{k,l}\big|_{G_{i,j}}\colon G_{i,j}\rightarrow
G_{k,l}$ is a bijective map. Now the compositions
$\varphi_{i,j}^{k,l}\big|_{G_{i,j}}{\circ}\varphi_{k,l}^{i,j}\big|_{G_{k,l}}$ and
$\varphi_{k,l}^{i,j}\big|_{G_{k,l}}{\circ}
\varphi_{i,j}^{k,l}\big|_{G_{i,j}}$ are identity maps of the sets
$G_{i,j}$ and $G_{k,l}$, respectively, and hence we have that the
map $\varphi_{i,j}^{k,l}\big|_{G_{i,j}}\colon G_{i,j}\rightarrow
G_{k,l}$ is invertible to $\varphi_{k,l}^{i,j}\big|_{G_{k,l}}\colon
G_{k,l}\rightarrow G_{i,j}$. Since $\mathscr{B}_W$ is a
semitopological semigroup we conclude that
$\varphi_{i,j}^{k,l}\big|_{G_{i,j}}\colon G_{i,j}\rightarrow
G_{k,l}$ and $\varphi_{k,l}^{i,j}\big|_{G_{k,l}}\colon
G_{k,l}\rightarrow G_{i,j}$ are continuous maps and hence the map
$\varphi_{i,j}^{k,l}\big|_{G_{i,j}}\colon G_{i,j}\rightarrow
G_{k,l}$ is a homeomorphism.

$(iii)$ Formula (\ref{eq-3.1}) implies that $G_{i,i}$ and $G_{k,k}$
are semitopological subgroups of $\mathscr{B}_W$ with the induced
topology from $\mathscr{B}_W$ for all $i,k\in\mathbb{Z}$. Simple
verifications show that the map
$\varphi^{k,k}_{i,i}\big|_{G_{i,i}}\colon G_{i,i}\rightarrow
G_{k,k}$ is a topological isomorphism.

$(iv)$ Assertion $(iii)$ implies that for arbitrary
$i,k\in\mathbb{Z}$ the subspaces $G_{i,i}$ and $G_{k,k}$ with the
induced semigroup operation are topologically isomorphic subgroups
of $\mathscr{B}_W$ and hence the semitopological group $G$ is
correctly determined. Next we consider the map $f\colon
G=G_{0,0}\rightarrow G=G_{1,1}$ defined by the formula
$f(x)=x\cdot(1,e,1)$. Then by formula (\ref{eq-3.1}) we have that
\begin{equation*}
    (0,g,0)\cdot(1,e,1)=
    (1,f^{-1}_{1,0}\cdot\theta(g)\cdot f_{1,0}\cdot e,1)=
    (1,e^{-1}\cdot\theta(g)\cdot e\cdot e,1)=
    (1,\theta(g),1),
\end{equation*}
and since the translations in $\mathscr{B}_W$ are continuous we
conclude that $\theta$ is a continuous endomorphism of the
semitopological group $G$.

$(v)$ Since the left and right translations in a semitopological
semigroup are continuous maps and left and right translations by an
idempotent are retractions, Exercise~1.5.C from \cite{Engelking1989}
implies that $(i+1,e,i+1)\mathscr{B}_W$ and
$\mathscr{B}_W(j+1,e,j+1)$ are closed subsets in $\mathscr{B}_W$.
Hence there exists an open neighbourhood $W_{(i,g,j)}$ of the point
$(i,g,j)$ in $\mathscr{B}_W$ such that
\begin{equation*}
W_{(i,g,j)}\subseteq\mathscr{B}_W\setminus
\left((i+1,e,i+1)\mathscr{B}_W\cup \mathscr{B}_W(j+1,e,j+1)\right).
\end{equation*}
Since the semigroup operation in $\mathscr{B}_W$ is separately
continuous we conclude that there exists an open neighbourhood
$U_{(i,g,j)}$ of the point $(i,g,j)$ in $\mathscr{B}_W$ such that
\begin{equation*}
    U_{(i,g,j)}\subseteq W_{(i,g,j)}, \quad
    (i,e,i)\cdot U_{(i,g,j)}\subseteq W_{(i,g,j)}
    \quad \hbox{and} \quad
    U_{(i,g,j)}\cdot(j,e,j)\subseteq W_{(i,g,j)}.
\end{equation*}
Next we shall show that $U_{(i,g,j)}\subseteq
\bigcup\{G_{i-k,j-k}\colon k=0,1,2,3,\ldots\}$. Suppose the
contrary: there exists $(m,a,n)\in U_{(i,g,j)}$ such that
$(m,a,n)\notin\bigcup\{G_{i-k,j-k}\colon k=0,1,2,3,\ldots\}$. Then
we have that $m\leqslant i$, $n\leqslant j$ and $m-n\neq i-j$. If
$m-n>i-j$ then formula (\ref{eq-3.1}) implies that there exists
$u\in G$ such that
\begin{equation*}
(m,a,n)\cdot (j,e,j)=(m-n+j,u,j) \notin\mathscr{B}_W\setminus
(i+1,e,i+1)\mathscr{B}_W,
\end{equation*}
because $m-n+j>i-j+j=i$, and hence $(m,a,n)\cdot (j,e,j)\notin
W_{(i,s,j)}$. Similarly, if $m-n<i-j$ then formula (\ref{eq-3.1})
implies that there exists $v\in G$ such that
\begin{equation*}
(i,e,i)\cdot(m,a,n)=(i,v,n-m+i) \notin\mathscr{B}_W\setminus
\mathscr{B}_W(j+1,e,j+1),
\end{equation*}
because $n-m+i>j-i+i=j$, and hence $(i,e,i)\cdot(m,a,n)\notin
W_{(i,s,j)}$. This completes the proof of our assertion.

$(vi)$ The definition of an $I$-bisimple semigroup implies that
$E(S)$ is order isomorphic to $\mathbb{Z}$ under the reverse of the
usual order and hence $E(S)$ is a subsemigroup of $S$. Then we have
that $E(S)=\{(n,e,n)\colon n\in\mathbb{Z}\}$ (see \cite{Warne1968}).
We fix an arbitrary $(i,e,i)\in E(S)$. Since translations on
$(i,e,i)$ in $S$ are continuous retractions Theorem~1.4.1 of
\cite{Engelking1989} implies that the set $\{x\in S\colon
x\cdot(i-1,e,i-1)=(i-1,e,i-1)\}$ is closed in $S$, and
Exercise~1.5.C from \cite{Engelking1989} implies that $(i+1,e,i+1)S$
is a closed subset in $S$ too. This implies that $(i,e,i)$ is an
isolated point of $E(S)$ with the induced from $S$ topology. This
completes the proof of our assertion.
\end{proof}

\begin{theorem}\label{theorem-3.3}
Let $S$ be a regular $I$-bisimple semitopological semigroup. If $S$
has a maximal compact subgroup then the following statements hold:
\begin{itemize}
  \item[$(i)$] $S$ is topologically isomorphic to
   $\mathscr{B}_W=\mathbb{Z}\times G\times\mathbb{Z}$ with the
   product topology;
  \item[$(ii)$] $S$ is a locally compact topological inverse
   semigroup.
\end{itemize}
\end{theorem}

\begin{proof}
$(i)$ By item $(i)$ of Theorem~\ref{theorem-3.2} we have that the
semitopological semigroup $S$ is topologically isomorphic to a
semitopological semigroup $\mathscr{B}_W=\mathbb{Z}\times
G\times\mathbb{Z}$. It is obvious to show that for arbitrary
$i,j\in\mathbb{Z}$ the $\mathscr{H}$-class $G_{i,j}$ of
$\mathscr{B}_W$ is an open subset in $\mathscr{B}_W$. We fix an
arbitrary $(i,g,j)\in G_{i,j}$. Then by
Theorem~\ref{theorem-3.2}$(v)$ there exists an open neighbourhood
$U_{(i,g,j)}$ of the point $(i,g,j)$ in $\mathscr{B}_W$ such that
$U_{(i,g,j)}\subseteq \bigcup\left\{G_{i-k,j-k}\colon
k=0,1,2,3,\ldots\right\}$. Since the semitopological semigroup $S$
has a maximal compact subgroup, Theorem~\ref{theorem-3.2}$(ii)$
implies that every $\mathscr{H}$-class $G_{m,n}$ of $\mathscr{B}_W$
is a compact subset in $\mathscr{B}_W$. Then the separate continuity
of the semigroup operation in $\mathscr{B}_W$ and Theorem~1.4.1 of
\cite{Engelking1989} imply that $\left\{x\in
\mathscr{B}_W\colon x\cdot(i-1,e,i-1)\in G_{i-1,i-1}\right\}$ is a closed set
in $\mathscr{B}_W$. Therefore there exists an open
neighbourhood $V_{(i,g,j)}\subseteq U_{(i,g,j)}$ of the point
$(i,g,j)$ in $\mathscr{B}_W$ such that $V_{(i,g,j)}\subseteq
G_{i,j}$. This completes the proof of the statement.

$(ii)$ Statement $(i)$, Theorem~\ref{theorem-3.2}$(ii)$ and
Theorem~3.3.13 of \cite{Engelking1989} imply that $S$ is a locally
compact space. Then statement $(i)$, Corollary~3.3.10 from
\cite{Engelking1989} and Ellis Theorem (see Theorem~2 of
\cite{Ellis1957} or Theorem~1.18 of \cite[Vol.~1]{CHK}) imply that
every maximal subgroup $G_{n,n}$ of $\mathscr{B}_W$ is a topological
group. We put $G=G_{n,n}$ for some $n\in\mathbb{Z}$ with the induced
topology from $\mathscr{B}_W$. Assertion $(iii)$ of
Theorem~\ref{theorem-3.2} implies that the topological group $G$ is
correctly defined. Let $\mathfrak{B}_G$ be a base of the topology of
the topological group $G$. Then statement $(i)$ and assertion $(ii)$
of Theorem~\ref{theorem-3.2} imply that the family
\begin{equation*}
    \mathfrak{B}_{\mathscr{B}_W}=\left\{U_{i,j}\colon U\in
    \mathfrak{B}_G \hbox{ and } i,j\in\mathbb{Z}\right\},
\end{equation*}
where $U_{i,j}=\left\{(i,x,j)\colon x\in U\right\}\subseteq
G_{i,j}$, determines a base of the topology of the semitopological
semigroup $\mathscr{B}_W$.

Since $G$ is a topological group and $\theta\colon G\rightarrow G$
is a continuous homomorphism, we conclude that for arbitrary
integers $a,b,c,d$ with $b\geqslant c$, arbitrary $g,h\in G$ and any
open neighbourhood $W$ of the point $g\cdot
f^{-1}_{b-c,c}\cdot\theta^{b-c}(h)\cdot f_{b-c,d}$ in the
topological space $G$ there exist open neighbourhoods $W_{g}$ and
$W_{h}$ of the points $g$ and $h$ in $G$, respectively, such that
\begin{equation*}
    W_{g}\cdot f^{-1}_{b-c,c}\cdot\theta^{b-c}(W_{h})\cdot
    f_{b-c,d}\subseteq W.
\end{equation*}
Then in the case when $b\geqslant c$ we have that
\begin{equation*}
    (a,W_{g},b){\cdot}(c,W_{h},d){\subseteq}
    (a, W_{g}\cdot f^{-1}_{b-c,c}\cdot\theta^{b-c}(W_{h})\cdot
    f_{b-c,d},d-c+b){\subseteq} (a, W,d-c+b).
\end{equation*}
Similarly, the continuity of the group operation in $G$ and the
continuity of the homomorphisms $\theta$ imply that for arbitrary
integers $a,b,c,d$ with $b\leqslant c$, arbitrary $g,h\in G$ and any
open neighbourhood $U$ of
$f^{-1}_{c-b,a}\cdot\theta^{c-b}(g)\cdot f_{c-b,b}\cdot h$ in the
topological space $G$ there exist open neighbourhoods $U_{g}$ and
$U_{h}$ of the points $g$ and $h$ in $G$, respectively, such that
\begin{equation*}
    f^{-1}_{c-b,a}\cdot\theta^{c-b}(U_{g})\cdot f_{c-b,b}\cdot
    U_{h}\subseteq U.
\end{equation*}
Then in the case when $b\leqslant c$ we have that
\begin{equation*}
    (a,U_{g},b){\cdot}(c,U_{h},d)\subseteq
    (a-b+c, f^{-1}_{c-b,a}\cdot\theta^{c-b}(U_{g})\cdot
    f_{c-b,b}\cdot U_{h},d)\subseteq (a-b+c,U,d).
\end{equation*}
Hence the semigroup operation is continuous in $\mathscr{B}_W$.

Also, since the inversion in $G$ is continuous we have that for
every element $g$ of $G$ and any open neighbourhood $W_{g^{-1}}$ of
its inverse $g^{-1}$ in $G$ there exists open neighbourhood $U_{g}$
of $g$ in $G$ such that $\left(U_{g}\right)^{-1}\subseteq
W_{g^{-1}}$. Then we get that
$
    (a,U_{g},b)^{-1}\subseteq (b,W_{g^{-1}},a),
$
for arbitrary integers $a$ and $b$. This completes the proof that
$\mathscr{B}_W$ is a topological inverse semigroup.
\end{proof}

If $S$ is a topological inverse semigroup then the maps
$\mathfrak{l}\colon S\rightarrow E(S)$ and $\mathfrak{r}\colon
S\rightarrow E(S)$ defined by the formulae $\mathfrak{l}(x)=x\cdot
x^{-1}$ and $\mathfrak{r}(x)=x^{-1}\cdot x$ are continuous. Hence
Theorem~\ref{theorem-3.2} implies the following corollary:

\begin{corollary}\label{corollary-3.4}
Let $S$ be a regular $I$-bisimple topological inverse semigroup.
Then every $\mathscr{H}$-class of $S$ is a closed-and-open subset of
$S$.
\end{corollary}

A topological space $X$ is called \emph{Baire} if for each sequence
$A_1, A_2,\ldots, A_i,\ldots$ of nowhere dense subsets of $X$ the
union $\displaystyle\bigcup_{i=1}^\infty A_i$ is a co-dense subset
of $X$~\cite{Engelking1989}.

Since every Hausdorff Baire topology on a countable topological
group is discrete, Corollary~\ref{corollary-3.4} implies the
following:

\begin{corollary}\label{corollary-3.5}
Every regular $I$-bisimple countable Hausdorff Baire topological
inverse semigroup is discrete.
\end{corollary}

A Tychonoff space $X$ is called \emph{\v{C}ech complete} if for
every compactification $cX$ of $X$ the remainder $cX\setminus c(X)$
is an $F_\sigma$-set in $cX$~\cite{Engelking1989}. Since every
\v{C}ech complete space (and hence every locally compact space) is
Baire, Corollary~\ref{corollary-3.5} implies the following:

\begin{corollary}\label{corollary-3.6}
Every regular $I$-bisimple countable Hausdorff \v{C}ech complete
(locally compact) topological inverse semigroup is discrete.
\end{corollary}

The following example implies that there exists a Hausdorff locally
compact zero-dimensional $I$-bisimple topological semigroup $S$ with
locally compact (discrete) maximal subgroup $G$ such that $S$ is not
topologically isomorphic to $\mathscr{B}_W=\mathbb{Z}\times
G\times\mathbb{Z}$ with the product topology and hence $S$ is not a
topological inverse semigroup.

\begin{example}\label{example-3.7}
Let $Z$ be the additive group of integers and $\theta\colon
Z\rightarrow Z$ be an annihilating homomorphism, i.e., $\theta(m)=e$
is the unity of $Z$ for every $m\in Z$. We put
$\mathscr{B}(Z,\mathbb{Z})$ be the $\mathbb{Z}$-Bruck extension of
the group $Z$. Then Theorem~\ref{theorem-3.1} implies that
$\mathscr{B}(Z,\mathbb{Z})$ is an $I$-bisimple semigroup.

We determine the topology $\tau$ on $\mathscr{B}(Z,\mathbb{Z})$ in
the following way:
\begin{itemize}
  \item[$(i)$] all non-idempotent elements of the semigroup
   $\mathscr{B}(Z,\mathbb{Z})$ are isolated points in
   $\left(\mathscr{B}(Z,\mathbb{Z}),\tau\right)$; \qquad and
  \item[$(ii)$] the family $\mathfrak{B}_{(i,e,j)}=
   \big\{U^n_{i,j}\colon i,j\in\mathbb{Z}, n\in\mathbb{Z}\big\}$,
   where $U^n_{i,j}=\{(i,e,j)\}\cup\{(i-1,k,j-1)\colon k\geqslant
   n\}$, is a base of the topology $\tau$ at the point
   $(i,e,j)\in\mathscr{B}(Z,\mathbb{Z})$, $i,j\in\mathbb{Z}$.
\end{itemize}

Simple verifications show that $\tau$ is a Hausdorff locally compact
zero-dimensional topology on $\mathscr{B}(Z,\mathbb{Z})$. Later we
shall prove that $\tau$ is a semigroup topology on
$\mathscr{B}(Z,\mathbb{Z})$.

We remark that the semigroup operation on
$\mathscr{B}(Z,\mathbb{Z})$ is defined by the formula
\begin{equation*}
    (i,g,j)\cdot(m,h,n)=
\left\{
  \begin{array}{ll}
    (i-j+m,h,n),        & \hbox{if }~j<m; \\
    (i,g\cdot h,n),     & \hbox{if }~j=m; \\
    (i,g,n-m+j),        & \hbox{if }~j>m,\\
  \end{array}
\right.
\end{equation*}
for arbitrary $i,j,m,n\in\mathbb{Z}$ and $g,h\in Z$. Since all
non-idempotent elements of the semigroup $\mathscr{B}(Z,\mathbb{Z})$
are isolated points in
$\left(\mathscr{B}(Z,\mathbb{Z}),\tau\right)$, it is sufficient to
show that the semigroup operation on
$\left(\mathscr{B}(Z,\mathbb{Z}),\tau\right)$ is continuous in the
following cases:
\begin{equation*}
    {\textbf{a)}}~(i,g,j)\cdot(m,e,n); \qquad
    {\textbf{b)}}~(i,e,j)\cdot(m,g,n); \qquad
    {\textbf{c)}}~(i,e,j)\cdot(m,e,n),
\end{equation*}
where $e$ is the unity of $G$ and $g\in G\setminus\{e\}$.

Then we have that in case \textbf{a)}:
\begin{itemize}
  \item[$(1)$] if $j{<}m{-}1$ then $(i,g,j){\cdot}(m,e,n){=}(i{-}j{+}m,e,n)$ and
   $\{(i,g,j)\}{\cdot} U_{m,n}^k\subseteq U_{i-j+m,n}^{k}$;

  \item[$(2)$] if $j{=}m{-}1$ then $(i,g,j){\cdot}(m,e,n){=}(i{+}1,e,n)$ and
   $\{(i,g,j)\}{\cdot} U_{m,n}^k\subseteq U_{i+1,n}^{k+g}$;

  \item[$(3)$] if $j\geqslant m$ then $(i,g,j)\cdot(m,e,n)=(i,g,n-m+j)$ and
   $\{(i,g,j)\}\cdot U_{m,n}^k\subseteq \{(i,g,n-m+j)\}$,
\end{itemize}
in case \textbf{b)}:
\begin{itemize}
  \item[$(1)$] if $j\leqslant m$ then
   $(i,e,j)\cdot(m,g,n)=(i-j+m,g,n)$ and
   $U_{i,j}^k\cdot\{(m,g,n)\} \subseteq \{(i-j+m,g,n)\}$;

  \item[$(2)$] if $j=m+1$ then $(i,e,j)\cdot(m,g,n)=(i,e,n+1)$ and
   $U_{i,j}^k\cdot\{(m,g,n)\} \subseteq U_{i,n+1}^{k+g}$;

  \item[$(2)$] if $j{>}m{+}1$ then $(i,e,j){\cdot}(m,g,n){=}(i,e,n-m+j)$ and
   $U_{i,j}^k\cdot\{(m,g,n)\} \subseteq U_{i,n-m+j}^{k}$,
\end{itemize}
and in case \textbf{c)}:
\begin{itemize}
  \item[$(1)$] if $j<m$ then
   $(i,e,j)\cdot(m,e,n)=(i-j+m,e,n)$ and
   $U_{i,j}^k\cdot U_{m,n}^l\subseteq U^l_{i-j+m,n}$;
  \item[$(2)$] if $j=m$ then $(i,e,j)\cdot(m,e,n)=(i,e,n)$ and
   $U_{i,j}^k\cdot U_{m,n}^l \subseteq U_{i,n }^{k+l}$;
  \item[$(3)$] if $j>m$ then $(i,e,j)\cdot(m,e,n)=(i,e,n-m+j)$ and
   $U_{i,j}^k\cdot U_{m,n}^l\subseteq U^k_{i,n-m+j}$,
\end{itemize}
for arbitrary integers $k$ and $l$. Hence
$\left(\mathscr{B}(Z,\mathbb{Z}),\tau\right)$ is a topological
semigroup. It is obvious that the inversion in
$\left(\mathscr{B}(Z,\mathbb{Z}),\tau\right)$ is not continuous.
\end{example}

\begin{remark}\label{remark-3.8}
\begin{itemize}
  \item[$(1)$] We observe that the similar propositions to
   Theorems~\ref{theorem-3.2} and \ref{theorem-3.3},
   Corollaries~\ref{corollary-3.4}, \ref{corollary-3.5} and
   \ref{corollary-3.6} hold for $\omega$-bisimple (semi)topological
   semigroups as topological Bruck-Reilly extensions.

  \item[$(2)$] Also Example~\ref{example-3.7} shows that there
   exists a Hausdorff locally compact zero-dimensional
   $\omega$-bisimple topological semigroup $S$ with a locally compact
   (discrete) maximal subgroup $G$ such that $S$ is not
   topologically isomorphic to the Bruck-Reilly extension
   with the product topology and hence $S$ is not
   a topological inverse semigroup.

  \item[$(3)$] The statement of Theorem~\ref{theorem-3.3} is true in the case when the subsemigroup $C(S)=\{(i,g,i)\colon i\in\mathbb{Z} \hbox{ and } g\in G\}$ is weakly uniform (the definition of a weakly uniform topological semigroup see in \cite{Shneperman1968}). In this case we have that inversion in $C(S)$ is continuous (see \cite{Demenchuk1970} and \cite{Demenchuk1971}) and hence by Proposition~\ref{proposition-2.3} we get that  every $\mathscr{H}$-class of $S$ is an open-and-closed subset of $S$. This implies that the inversion in $S$ is continuous, too.
\end{itemize}
\end{remark}

The following example implies that there exists a Hausdorff locally compact zero-dimensional $I$-bisimple semitopological semigroup $S$ with continuous inversion and locally compact (discrete) maximal subgroup $G$ such that $S$ is not
topologically isomorphic to $\mathscr{B}_W=\mathbb{Z}\times G\times\mathbb{Z}$ with the product topology and hence $S$ is not a topological inverse semigroup.

\begin{example}\label{example-3.9}
Let $Z$ be the additive group of integers and $\theta\colon Z\rightarrow Z$ be an annihilating homomorphism.

We determine the topology $\tau$ on $\mathscr{B}(Z,\mathbb{Z})$ in
the following way:
\begin{itemize}
  \item[$(i)$] all non-idempotent elements of the semigroup
   $\mathscr{B}(Z,\mathbb{Z})$ are isolated points in
   $\left(\mathscr{B}(Z,\mathbb{Z}),\tau\right)$; \qquad and

  \item[$(ii)$] the family $\mathfrak{B}_{(i,e,j)}=
   \big\{U^{m,n}_{i,j}\colon i,j\in\mathbb{Z}, m,n\in\mathbb{Z}\big\}$,
   where
 \begin{equation*}
   U^{m,n}_{i,j}=\{(i,e,j)\}\cup\{(i-1,k,j-1)\colon k\leqslant-n\}\cup \{(i-1,k,j-1) \colon k\geqslant n\},
 \end{equation*}
  is a base of the topology $\tau$ at the point
   $(i,e,j)\in\mathscr{B}(Z,\mathbb{Z})$, $i,j\in\mathbb{Z}$.
\end{itemize}

Simple verifications show that $\tau$ is a Hausdorff locally compact
zero-dimensional topology on $\mathscr{B}(Z,\mathbb{Z})$. The proof of the separate continuity of semigroup operation and the continuity of inversion in
$(\mathscr{B}(Z,\mathbb{Z}),\tau)$ is similar to Example~\ref{example-3.7}.
\end{example}

\begin{remark}\label{remark-3.10}
Example~\ref{example-3.9} shows that there exists a Hausdorff locally compact zero-dimensional $\omega$-bisimple semitopological semigroup $S$ with continuous inversion and a locally compact (discrete) maximal subgroup $G$ such that $S$ is not
topologically isomorphic to the Bruck-Reilly extension with the product topology and hence $S$ is not a topological inverse semigroup.
\end{remark}

\section*{Acknowledgements}

The research of the third-named author was carried out with the
support of the Estonian Science Foundation and co-funded by Marie
Curie Action, grant ERMOS36.

The authors are grateful to the referee for several useful comments
and suggestions.
%%%%%%%%%%%%%%%%%%%%%%%%%%%%%%%%%%%%%%%%%%%%%%%%%%%%%%%%%%%%%%%

\end{document}